\documentclass[reqno]{amsart}
\usepackage{amsmath,amsfonts,amsthm,amssymb}
\usepackage{enumitem}

\usepackage[normalem]{ulem}

\usepackage{bbm}
\usepackage{hyperref}

\allowdisplaybreaks   

\theoremstyle{plain}
\newtheorem{de}{Definition}[section]
\newtheorem{lem}[de]{Lemma}
\newtheorem{prop}[de]{Proposition}
\newtheorem{cor}[de]{Corollary}
\newtheorem{thm}[de]{Theorem}

\theoremstyle{definition}
\newtheorem{rem}[de]{Remark}

\numberwithin{equation}{section}

\newcommand{\one}{\mathbbm{1}}
\newcommand{\eul}{e}
\newcommand{\imu}{\mathrm{i}}
\newcommand{\dd}{\,\mathrm{d}}

\renewcommand{\Re}{\operatorname{Re}}
\newcommand{\supp}{\operatorname{supp}}

\newcommand{\R}{{\mathbb{R}}}
\newcommand{\N}{{\mathbb{N}}}

\newcommand{\C}{{\mathbb{C}}}
\newcommand{\Z}{{\mathbb{Z}}}

\newcommand{\PP}{{\mathbb{P}}}

\newcommand{\cL}{{\mathcal{L}}}
\newcommand{\cF}{{\mathcal{F}}}
\newcommand{\cA}{{\mathcal{A}}}

\newcommand{\cP}{{\mathcal{P}}}

\newcommand{\Schw}{\mathcal{S}}

\newcommand{\om}{\omega}
\newcommand{\si}{\sigma}

\newcommand{\ep}{\varepsilon}

\newcommand{\linwave}{v_{\mathrm{li}}}
\newcommand{\linschr}{u_{\mathrm{li}}}

\newcommand{\nlwave}{v_{\mathrm{nl}}}
\newcommand{\nlschr}{u_{\mathrm{nl}}}

\newcommand{\schrsp}{X_T^\si}
\newcommand{\wavesp}{Y_T^\si}

\newcommand{\btwt}[3]{\dot{B}^{#1}_{#2,#3,2}}

\newcommand{\bdyop}{\Omega_{\mathrm{b}}}
\newcommand{\resf}{\omega_{\mathrm{r}}}

\begin{document}
 \title[Randomized final-state problem for the Zakharov system in 3D]{Randomized final-state problem for the Zakharov system in dimension three}
 
 \author[M.~Spitz]{Martin Spitz}
 \address[M.~Spitz]{Fakult\"at f\"ur Mathematik, Universit\"at Bielefeld, Postfach 10 01 31, 33501
  Bielefeld, Germany}
\email{mspitz@math.uni-bielefeld.de}

 \begin{abstract}
 	We consider the final-state problem for the Zakharov system in the energy space in three space dimensions. For $(u_+, v_+) \in H^1 \times L^2$ without any size restriction, symmetry assumption or additional angular regularity, we perform a physical-space randomization on $u_+$ and an angular randomization on $v_+$ yielding random final states $(u_+^\omega, v_+^\omega)$. We obtain that for almost every $\omega$, there is a unique solution of the Zakharov system scattering to the final state $(u_+^\omega, v_+^\omega)$. The key ingredient in the proof is the use of time-weighted norms and generalized Strichartz estimates which are accessible due to the randomization.
 \end{abstract}
 
 \maketitle 
 
 \section{Introduction and main results}
 
 	\subsection{The Zakharov system}
 	
 	The Zakharov system is a model in plasma physics to describe Langmuir waves in a non- or weakly magnetized plasma. Langmuir waves are rapid oscillations of the electric field in the plasma. The scalar version of the Zakharov system is given by
 	\begin{equation}
 		\label{eq:ZakharovSystem}
 		 \begin{aligned}
         &\imu \partial_t u + \Delta u = V u, \\
         &\frac{1}{\alpha^2}\partial_t^2 V -  \Delta V =  \Delta |u|^2.
        \end{aligned}
 	\end{equation}
 	 Here, $V \colon \R \times \R^3 \rightarrow \R$ denotes the fluctuation of the ion density, $u \colon \R \times \R^3 \rightarrow \C$ the complex envelope of the electric field, and the fixed constant $\alpha > 0$ the ion sound speed. We refer to~\cite{Z1972,SS1999, T2007, CEGT2007} for the physical background and the derivation of this system.
 	
 	For the purpose of this article it is more convenient to use the first order reformulation of the Zakharov system. Setting $v = V - \imu \alpha^{-1} |\nabla|^{-1} \partial_t V$, system~\eqref{eq:ZakharovSystem} is equivalent to
 	\begin{equation}
 	 \label{eq:ZakharovSystemFirstOrder}
        \begin{aligned}
         &\imu \partial_t u + \Delta u = \Re(v) u, \\
         &\imu \partial_t v + \alpha |\nabla| v = - \alpha |\nabla| |u|^2.
        \end{aligned}
 	\end{equation}
 	
 	The Zakharov system has been extensively studied in the literature. Local wellposedness was shown for example in~\cite{BC1996, GTV1997, BH2011}, see also~\cite{BHHT2009, BGHN2015, CHN2019} for other dimensions. We particularly note that in $d \geq 4$ the recent work~\cite{CHN2019} gave a complete answer to the question of local wellposedness by determining the optimal range for the parameters $(s,l) \in \R^2$ such that~\eqref{eq:ZakharovSystemFirstOrder} is locally wellposed for initial data in $H^s(\R^d) \times H^l(\R^d)$.
 	
 	The Zakharov system preserves the Schr{\"o}dinger mass $M(u)$ and the energy $E_Z(u,v)$ given by
 	\begin{align*}
 		M(u) = \int_{\R^3} \frac{1}{2} |u|^2 \dd x, \qquad E_Z(u,v) =  \int_{\R^3} \frac{1}{2}|\nabla u|^2 + \frac{1}{4} |v|^2 + \frac{1}{2}\Re(v) |u|^2 \dd x.
 	\end{align*}
 	The \emph{energy space} $H^1(\R^3) \times L^2(\R^3)$ is thus of particular interest. Concerning the long time behavior the energy is of limited use because of the indefinite term $\Re(v) |u|^2$. However, if the $H^1$-norm of the Schr{\"o}dinger component of the data is small, the energy gives a priori control over the energy norm of the solution. This was used in~\cite{BC1996} to conclude global wellposedness for data $(u_0, v_0)$ from the energy space if $\|u_0\|_{H^1}$ is sufficiently small.
 	
 	There is a close connection between the Zakharov system and the cubic focusing nonlinear Schr{\"o}dinger equation 
	\begin{equation}
	\label{eq:CubicFocusingSchr}
		\imu \partial_t u + \Delta u = -|u|^2 u
	\end{equation}	 	
 	since the latter arises as the subsonic limit ($\alpha \rightarrow \infty$) of the former, see~\cite{KPV1995, MN2008, OT1992, SW1986}. Let $Q$ denote the ground state of~\eqref{eq:CubicFocusingSchr}, i.e. the unique positive radial solution of
 	\begin{align*}
 		- \Delta Q + Q = Q^3,
 	\end{align*}
 	minimizing the action
 	\begin{align*}
 		J(\varphi) = \int_{\R^3} \frac{1}{2} |\nabla \varphi|^2 + \frac{1}{2} |\varphi|^2 - \frac{1}{4} |\varphi|^4 \dd x = E_S(\varphi) +  M(\varphi),
 	\end{align*}
 	where $E_S$ denotes the energy for the nonlinear Schr{\"o}dinger equation~\eqref{eq:CubicFocusingSchr}.
 	We refer to~\cite{HR2008} for more information on $Q$. This ground state gives rise to a radial standing wave solution $(u(t), v(t)) = (\eul^{\imu t} Q, - Q^2)$ of the Zakharov system~\eqref{eq:ZakharovSystemFirstOrder}. In fact, one can construct a whole family of radial standing waves of~\eqref{eq:ZakharovSystemFirstOrder} from $Q$, see~\cite{GNW2013}.
 	
 	Concerning the long-time behavior of the Zakharov system, the standing wave solution particularly implies that not every solution in the energy space will scatter and the ground state provides a natural threshold for scattering. Moreover, in~\cite{M1996} it was shown that radial solutions with negative energy blow up in finite or in infinite time. See also~\cite{GNW2013} for blow-up results.
 	
 	On the other hand, several positive results concerning the asymptotic behavior of solutions to the Zakharov system~\eqref{eq:ZakharovSystemFirstOrder} have been established. We first review the scattering problem, i.e. the question, for which initial data $(u_0, v_0)$ there are $(u_+, v_+)$ in the energy space such that
 	\begin{equation}
 	\label{eq:ScatteringForZakharov}
 		\| u(t) - \eul^{\imu t \Delta} u_+\|_{H^1} + \| v(t) - \eul^{\imu \alpha t |\nabla|} v_+ \|_{L^2} \longrightarrow 0
 	\end{equation}
 	as $t \rightarrow \infty$, where $(u,v)$ denotes the solution of~\eqref{eq:ZakharovSystemFirstOrder} with data $(u_0, v_0)$. In~\cite{GN2014} this question was answered positively for small radially symmetric data in the energy space. This result was then extended to radially symmetric data below the ground state in~\cite{GNW2013}. The assumption of radial symmetry was weakened in~\cite{GLNW2014, G2016}, where scattering for small data in the energy space with additional angular regularity was shown.
 	We further note that in~\cite{HPS2013} a scattering result for regular and spatially decaying data was developed.
 	
 	A counterpart to the scattering problem as described above is the \emph{final-state problem}. Given $(u_+, v_+) \in H^1(\R^3) \times L^2(\R^3)$ one asks if there is a (unique) solution $(u,v)$ of~\eqref{eq:ZakharovSystemFirstOrder} scattering to $(u_+, v_+)$, i.e. satisfying~\eqref{eq:ScatteringForZakharov}. This question was studied in~\cite{OT1993, S2004, GV2006}. In these works positive answers were given not in the energy space but for more regular data satisfying several additional conditions.
 	
 	In this article we study the final-state problem for~\eqref{eq:ZakharovSystemFirstOrder} in the energy space without imposing any conditions on the size, radial symmetry or angular regularity of the data but using randomization instead.
 	
 	Before we state our main results Theorem~\ref{thm:MainResult} and Corollary~\ref{cor:GlobalSolution} below, we introduce the randomization procedures we will employ.

 	\subsection{Randomization}
 	
 	Since the seminal works~\cite{B1994, B1996} and~\cite{BT2008I, BT2008II} there has been large interest in random dispersive partial differential equations. One line of research is to study the question if in a supercritical setting, after randomizing the initial data, one still obtains local wellposedness, global wellposedness, or scattering almost surely. We refer to~\cite{BOP2015, BOP2019, DLM2019, KMV2019, LM2014, DLM2020, B2020, B2019} and the references therein for exemplary results in this direction for the Schr{\"o}dinger and the wave equation.
 	
 	We note that there are several possibilities how to randomize the data. While on compact manifolds the data was randomized with respect to a basis of eigenfunctions of the differential operator in~\cite{BT2008I,BT2008II}, on the full space most of the references above apply a Wiener randomization. Here the data is randomized with respect to a unit-scale decomposition of frequency space.
 	
 	In the recent work~\cite{M2019}, the author introduced a novel randomization with respect to a unit-scale decomposition of physical space, see Subsection~\ref{sss:RandPhysSpace} below for details. This was used to improve upon the known deterministic results for the final-state problem for the mass-subcritical NLS in $L^2$ almost surely. Roughly speaking, the physical-space randomization gives access to the dispersive estimate for linear solutions of the Schr{\"o}dinger equation although the data only belongs to $L^2$. In~\cite{NY2019} the authors observed that this dispersive decay can be used to study the final-state problem in time-weighted spaces, improving on the results in~\cite{M2019}.
 	
 	Another randomization was recently introduced in~\cite{BK2019}. Here a randomization with respect to the angular variable (see Subsection~\ref{sss:RandAngVar} for details) was combined with a randomization in the radial variable and a Wiener randomization in frequency space. This randomization was then applied to a wave maps type nonlinear wave equation with supercritical data.
 	
 	In the following we apply the physical-space randomization to the Schr{\"o}dinger final data and the angular randomization to the wave final data in order to study the final-state problem for the Zakharov system. The details of these randomization procedures are provided in the next two subsections. We discuss the advantages of this choice of randomization after the statement of the main result.
 	
 		\subsubsection{Randomization in physical space}
 		\label{sss:RandPhysSpace}
 	We first introduce the physical-space randomization for the Schr{\"o}dinger final data, following~\cite{M2019}.
 	
 	Fix a non-negative $\phi \in C^\infty_c(\R^3)$ such that $\phi(x) = 1$ for $|x| \leq 1$ and $\phi(x) = 0$ for $|x| \geq 2$. We then obtain a smooth partition of unity $\{\psi_k \colon k \in \Z^3\}$ by setting
 	\begin{equation}
 	\label{eq:DefPartitionUnityPhysicalSpaceRandomization}
 		\psi_k(x) = \frac{\phi(x-k)}{\sum_{l \in \Z^3} \phi(x-l)}
 	\end{equation}
 	for all $k \in \Z^3$.
 	
 	Next we take a sequence $(X_k)_{k \in \Z^3}$ of independent, real-valued, mean-zero random variables on a probability space $(\Omega, \cA, \PP)$ and denote their distributions by $\mu_k$. We assume that there is a constant $c > 0$ such that
 	\begin{equation}
 		\label{eq:ConditionRandomVariables}
 		\Big| \int_{\R} \eul^{\gamma x} \dd \mu_k(x) \Big| \leq e^{c \gamma^2}
 	\end{equation}
 	for all $\gamma \in \R$ and $k \in \Z^3$. For instance, one can take a sequence of independent, mean-zero Gaussian random variables with uniformly bounded variance. Another example is the case where the $\mu_k$ are compactly supported.
 	
 	For any $f \in L^2(\R^3)$ we then define the physical-space randomization $f^\om$ of $f$ by
 	\begin{equation}
 		\label{eq:DefPhysicalSpaceRandomization}
 			f^\om(x) = \sum_{k \in \Z^3} X_k(\om) \psi_k(x) f(x),
 	\end{equation}
 	which is understood as a limit in $L^2(\Omega, L^2(\R^3))$.
 	
 	\subsubsection{Randomization in the angular variable}
 	\label{sss:RandAngVar}
 	We next present the randomization in the angular variable closely following~\cite{BK2019}, where this randomization was introduced.
 	We start by recalling that the eigenfunctions of the Laplacian on the sphere are the spherical harmonics of degree $k$, i.e. the restrictions to $S^2$ of the homogeneous harmonic polynomials of degree $k$. 
 	We denote the space of these functions by $E_k$. The dimension of $E_k$ is given by
 	\begin{align*}
 		N_k = \binom{k+2}{2} - \binom{k}{2} = 2k + 1.
 	\end{align*}
 	We fix an orthonormal frame
 	\begin{align*}
 		\{ b_{k,l} \in L^2(S^2) \colon l = 1, \ldots, N_k, \, k \in \N_0\}
 	\end{align*}
 	of $L^2(S^2)$, consisting of eigenfunctions of $\Delta_{S^2}$, with the property that there is a constant $C > 0$ such that
 	\begin{equation}
 		\label{eq:LqBoundGoodFrame}
 		\|b_{k,l}\|_{L^q(S^2)} \leq \begin{cases} 
 										C \sqrt{q} \quad &\text{if } q < \infty, \\
 										C \sqrt{\log k} &\text{if } q = \infty
 									\end{cases}
 	\end{equation}
 	for all $l \in \{1, \ldots, N_k\}$, $k \in \N$, and $q \in [2,\infty]$. The existence of such a frame follows from Th{\'e}or{\`e}me~6 and~Proposition~3.2 in \cite{BL2013}, see also~\cite[Theorem~1.1]{BK2019} and~\cite{BL2014}. Following~\cite{BK2019}, we call a frame $\{b_{k,l} \colon l=1,\ldots, N_k, \, k \in \N_0\}$ as above a \emph{good frame}.
 	
 	Next take a function $f \in L^2(\R^3)$. We first rescale the Littlewood-Paley blocks to frequency $1$ setting 
 	\begin{align*}
 		g_m = (P_m f)(2^{-m} \cdot)
 	\end{align*}
 	for every $m \in \Z$. After passage to polar coordinates we expand the Fourier transform of $g_m$ in terms of the good frame, which yields
 	\begin{equation}
 		\label{eq:gmhatInGoodFrame}
 		\hat{g}_m(\rho \theta) = \sum_{k = 0}^\infty \sum_{l = 1}^{N_k} \hat{c}^m_{k,l}(\rho) b_{k,l}(\theta)
 	\end{equation}
 	for every $m \in \Z$. Theorem~3.10 in~\cite{SW1971} thus gives the representation
 	\begin{equation}
 		\label{eq:gmInGoodFrame}
 		g_m(r \theta) = \sum_{k = 0}^\infty \sum_{l = 1}^{N_k} a_k  r^{-\frac{1}{2}} b_{k,l}(\theta) \int_0^\infty \hat{c}^m_{k,l}(\rho) J_{\frac{2k+1}{2}}(r\rho) \rho^{\frac{3}{2}} \dd \rho  
 	\end{equation}
 	of $g_m$ in the good frame, where $a_k = (2 \pi)^{\frac{3}{2}} \imu^k$ and the Bessel function $J_\mu$ is defined as
 	\begin{align*}
 		J_\mu(t) = \frac{(\frac{t}{2})^\mu}{\Gamma(\frac{2\mu+1}{2}) \Gamma(\frac{1}{2})}\int_{-1}^1 \eul^{\imu t s} (1-s^2)^{\frac{2\mu-1}{2}} \dd s
 	\end{align*}
 	for all $t > 0$ and $\mu > -\frac{1}{2}$. Note that by~\eqref{eq:gmhatInGoodFrame} and Plancherel's theorem we also have
 	\begin{equation}
 		\label{eq:L2Normgm}
 		\|g_m\|_{L^2(\R^3)}^2 \sim \sum_{k = 0}^\infty \sum_{l = 1}^{N_k} \| \hat{c}^m_{k,l} \|_{L^2(r^2 \dd r)}^2.
 	\end{equation}
 	For every $m \in \Z$ we now take a sequence of independent, real-valued, mean-zero random variables $(Y^m_{k,l})_{l \in \{1, \ldots, N_k\}, k \in \N_0}$ on a probability space $(\Omega,\cA,\PP)$ with the property that there is a constant $c > 0$ such that
 	\begin{align*}
 		\Big| \int_{\R} e^{\gamma x} \dd \mu^{m}_{k,l}(x) \Big| \leq e^{c \gamma^2}
 	\end{align*}
 	for all $\gamma \in \R$, $l \in \{1, \ldots, N_k\}$, and $k \in \N_0$, where $\mu^m_{k,l}$ denotes the distribution of $Y^m_{k,l}$. See Subsection~\ref{sss:RandPhysSpace} for examples of such random variables.
 	
 	In view of~\eqref{eq:gmInGoodFrame} we set
 	\begin{equation}
 		\label{eq:Randomizationgm}
 		g_m^{\omega}(r \theta) = \sum_{k = 0}^\infty \sum_{l = 1}^{N_k} a_k r^{-\frac{1}{2}} Y^m_{k,l}(\omega) b_{k,l}(\theta) \int_0^\infty \hat{c}^m_{k,l}(\rho) J_{\frac{2k+1}{2}}(r \rho) \rho^{\frac{3}{2}} \dd \rho  
 	\end{equation}
 	for every $m \in \Z$, where the right-hand side is understood as the limit in $L^2(\Omega,L^2(\R^3))$. Next we rescale the functions $g_m^\omega$ to frequency $2^m$ setting
 	\begin{align*}
 		f_m^\omega = g_m^\omega(2^m \cdot)
 	\end{align*}
 	for all $m \in \Z$.
 	Finally, we define the \emph{angular randomization} $f^\omega$ of $f$ as
 	\begin{equation}
 		\label{eq:DefAngularRandomization}
 		f^\omega = \sum_{m \in \Z} f_m^\omega,
 	\end{equation}
 	which is again understood as the limit in $L^2(\Omega,L^2(\R^3))$.

 	\subsection{Main results}
 	The main result of this article states that after applying a suitable randomization to any data from the energy space, the final-state problem for the Zakharov system has almost surely a unique solution. We refer to Subsection~\ref{ss:NotationFunctionSpaces} for the definition of the function spaces appearing in the theorem.	 	
 	
 	\begin{thm}
 		\label{thm:MainResult}
 		Let $0 < \nu \ll 1$, $u_+ \in H^1(\R^3)$ and $v_+ \in L^2(\R^3)$. Let $u_+^\omega$ denote the physical-space randomization of $u_+$ and $v_+^{\omega}$ the angular randomization of $v_+$.
 		
 		Then for almost all $\omega \in \Omega$ there exists a time $T \geq 1$ and a unique solution $(u,v) \in C([T,\infty),H^1(\R^3)) \times C([T,\infty),L^2(\R^3))$ such that
		\begin{equation} 		
		\label{eq:UniquenessCondition}
 		\begin{aligned}
 			&\| t^{\frac{1}{2}-\nu}\langle \nabla \rangle (u(t) - \eul^{\imu t \Delta} u_+^\omega)\|_{L^\infty_t L^2 \cap L^2_t \dot{B}^0_{6,2} \cap L^2_t \dot{B}^{\frac{1}{4}+\ep}_{(q(\ep),\frac{2}{1-\nu}),2} }  < \infty, \\
 			 &\| t^{\frac{1}{2}-\nu} (v(t) - \eul^{\imu \alpha t |\nabla|} v_+^\omega)\|_{L^\infty_t L^2} < \infty.
 		\end{aligned}
 		\end{equation}
 		In particular, the solution satisfies
 		\begin{align*}
 			\|u(t) - \eul^{\imu t \Delta} u_+^\omega\|_{H^1} + \|v(t) - \eul^{\imu \alpha t |\nabla|} v_+^\omega\|_{L^2} \longrightarrow 0
 		\end{align*}
 		as $t \rightarrow \infty$, i.e. the solution $(u,v)$ scatters in the energy space with final state $(u_+^\omega, v_+^\omega)$.
 	\end{thm}
 	
 	We emphasize that no size restriction and no radial symmetry or angular regularity assumption is imposed on $(u_+, v_+) \in H^1(\R^3) \times L^2(\R^3)$ in the above theorem.
 	
	\begin{rem}
		In the energy space there seems to be a loss of derivatives in the nonlinearity of the Schr{\"o}dinger part of the Zakharov system~\eqref{eq:ZakharovSystemFirstOrder}. We overcome this problem by employing the normal form transform from~\cite{GN2014}, see Subsection~\ref{ss:NormalForm} below. By a solution of~\eqref{eq:ZakharovSystemFirstOrder} we thus mean a solution of the integral equation~\eqref{eq:ZakharovSystemNormalForm}  arising from the normal form transform. We remark for the following discussion that the nonlinearity in~\eqref{eq:ZakharovSystemNormalForm} also contains boundary and cubic terms.
	\end{rem} 	
 	
 	The main difficulty in proving scattering results respectively global estimates for the Zakharov system is the weak decay of the wave component. In fact, consider the Schr{\"o}dinger part of the Zakharov system as a Schr{\"o}dinger equation with potential $v$ for a moment. As $v$ is the solution of a wave equation, the best possible decay of the potential in dimension three is $\|v(t)\|_{L^\infty} \sim t^{-1}$, which is just not integrable. This suggests that decay estimates alone are insufficient to solve the final-state problem.
 	
 	We now discuss the main ideas of the proof of Theorem~\ref{thm:MainResult}. We write $u(t) = \linschr^\omega(t) + \nlschr(t)$ and $v(t) = \linwave^\omega(t) + \nlwave(t)$, where $\linschr^\omega(t) = \eul^{\imu t \Delta} u_+^\omega$ and $\linwave^\omega(t) = \eul^{\imu \alpha t |\nabla|} v_+^\omega$. The physical-space randomization of the Schr{\"o}dinger data gives access to the dispersive estimate, which we want to use to control $\nlschr$ in time-weighted spaces as in~\cite{NY2019}. The additional decay then transfers to $\nlwave$ by the coupling of the system. The idea is that in the nonlinear terms we gain decay through the time weights which can then be used to close the estimates. This strategy works well for most of the nonlinear terms we have to deal with after the normal form transform. It also identifies the quadratic term $\linwave^\omega \nlschr$ as the most difficult one as we do not gain a time weight here and all the decay has to come from $\linwave^\omega$. We point out that the best possible decay of $t^{-1}$ is not improved by randomization. In particular, it is not apparent how a physical-space randomization of the wave data, which might seem natural at first sight, can be used to control the $\linwave^\omega \nlschr$ nonlinearity.
 	
 	At this point the angular randomization shows its benefit. We make the novel observation that randomizing with respect to a good frame does not only give an extended range of exponents for the Strichartz estimate but also arbitrarily high integrability in the spherical coordinate. In order to estimate $\linwave^\omega \nlschr$ it is thus sufficient to control $\nlschr$ in a spherically averaged norm. However, in spherically averaged spaces (deterministic) Strichartz estimates for the Schr{\"o}dinger equation hold for an extended range of exponents. This allows us to close the estimates.
 	
 	Since the other part $\nlwave \nlschr$ of the wave-Schr{\"o}dinger quadratic nonlinearity can be controlled by means of the additional time weight, it really is the interplay of the physical-space and the angular randomization which allows us to prove Theorem~\ref{thm:MainResult}.
 	
 	Finally, we can exploit energy conservation to extend the solution of~\eqref{eq:ZakharovSystemFirstOrder} found in Theorem~\ref{thm:MainResult} to $[0,\infty)$ for randomized data below the ground state $Q$. The extension relies on the local wellposedness theory for the Zakharov system from Proposition~3.1 in~\cite{GTV1997} and the unconditional uniqueness result from~\cite{MN2009}. 
 	
 	\begin{cor}
 		\label{cor:GlobalSolution}
 		Let $0 < \nu \ll 1$, $u_+ \in H^1(\R^3)$ and $v_+ \in L^2(\R^3)$. Let $u_+^\omega$ denote the physical-space randomization of $u_+$ and $v_+^{\omega}$ the angular randomization of $v_+$.
 		
 		Then there is a measurable set $\tilde{\Omega} \subseteq \Omega$ with $\PP(\tilde{\Omega}) = 1$ such that for all $\omega \in \tilde{\Omega}$ satisfying
 		\begin{align*}
 			(2\|\nabla u_+^\omega\|_{L^2}^2 + \|v_+^\omega\|_{L^2}^2)\|u_+^\omega\|_{L^2}^2< 8 E_S(Q) M(Q)
 		\end{align*}
 		there is a solution  $(u,v) \in C([0,\infty),H^1(\R^3)) \times C([0,\infty),L^2(\R^3))$ satisfying
 		\begin{align*}
 			\|u(t) - \eul^{\imu t \Delta} u_+^\omega \|_{H^1} + \|v(t) - \eul^{\imu \alpha t |\nabla|} v_+^\omega \|_{L^2} \longrightarrow 0
 		\end{align*}
 		as $t \rightarrow \infty$. It is unique among those solutions in $C([0,\infty),H^1(\R^3)\times L^2(\R^3))$ which satisfy~\eqref{eq:UniquenessCondition} on $[T,\infty)$ for some $T \geq 1$.
 	\end{cor}
 	
 	In Remark~\ref{rem:MeasureGlobal} we note that in the case of small Schr{\"o}dinger data we can quantify the measure of the set of $\omega$ such that the nonlinear solution scattering to the final state $(u_+^\omega, v_+^\omega)$ is global.
 	
 	The rest of the paper is organized as follows. In Section~\ref{sec:NormalFormPrelim} we introduce the function spaces we are using, discuss the normal form transform, and provide time-weighted Strichartz estimates. In Section~\ref{sec:LinearRandomizedEstimates} we show how the randomization procedures presented above lead to improved space-time estimates. Section~\ref{sec:MultilinearEstimates} contains the multilinear estimates which are needed to prove the main results in Section~\ref{sec:MainResults}.
 
 \section{Normal form and other preliminaries}
 \label{sec:NormalFormPrelim}
	In this section we fix some notation, in particular concerning the function spaces we use, review the normal form from~\cite{GN2014}, and provide time-weighted Strichartz estimates which we will need in the following.

 	\subsection{Notation and function spaces}
		\label{ss:NotationFunctionSpaces}
		
    Fix an even function $\eta_0 \in C_c^\infty(\R)$ such that $0 \leq \eta_0 \leq 1$, $\eta_0 = 1$ on the ball $B_{\frac{5}{4}}(0)$, and the support of $\eta_0$ is contained in $B_{\frac{8}{5}}(0)$. For every number $k \in \Z$ we define the functions $\rho_k(\xi) = \eta_0(|\xi|/2^k) - \eta_0(|\xi|/2^{k-1})$ and $\rho_{\leq k}(\xi) = \eta_0(|\xi|/2^k)$ on $\R^3$ as well as the corresponding Fourier multipliers 
    \begin{align*}
     P_k f = \cF^{-1}(\rho_k \hat{f}), \qquad P_{\leq k} f = \cF^{-1} (\rho_{\leq k} \hat{f}), \qquad P_{\geq k}f = \cF^{-1}(\hat{f} - \rho_{\leq k-1} \hat{f}),
    \end{align*}
    where $\hat{f} = \cF f$ denotes the Fourier transform of $f$. We further set $\tilde{P}_k = \sum_{|j - k| \leq 4} P_j$.
        
    In order to distinguish between different frequency interactions, we define the paraproduct-type operators
    \begin{align}
    \label{eq:DefParaproduct}
     &(f g)_{LH} = \sum_{k \in \Z} P_{\leq k - 5} f \, P_k g, \hspace{1.5em} (f g)_{HL} = (g f)_{LH}, \hspace{1.5em} (f g)_{HH} =  \sum_{\substack{|k_1 - k_2| \leq 4, \\ k_1, k_2 \in \Z}} P_{k_1} f \, P_{k_2} g
    \end{align}
    for any pair of functions $f$ and $g$. Anticipating the nonresonant interactions of the Zakharov system, we also set
    \begin{align}
    \label{eq:DefXLalphaL}
     &(f g)_{\alpha L} = \sum_{\substack{|k - \log_2 \alpha| \leq 4, \\ k \in \Z}} P_k f \, P_{\leq k - 5} g, \qquad (f g)_{XL} = \sum_{\substack{|k - \log_2 \alpha| > 4, \\ k \in \Z}} P_k f \, P_{\leq k - 5} g.
    \end{align}
    Note that $(f g)_{HL} = (f g)_{\alpha L} + (f g)_{X L}$. The sum of indices denotes the sum of the corresponding operators, e.g. $(f g)_{\alpha L + XL} = (f g)_{\alpha L} + (f g)_{X L}$. For later use we also introduce the abbreviation $(f g)_R = (f g)_{LH + HH + \alpha L}$.
    
    We further denote the symbol of the bilinear operator with index $\ast$ by $\mathcal{P}_{\ast}$, i.e.,
    \begin{equation*}
     \cF (f g)_\ast(\xi) = \int_{\R^3} \mathcal{P}_\ast(\xi - \eta, \eta) \hat{f}(\xi - \eta) \hat{g}(\eta) \dd \eta,
    \end{equation*}
    where $\ast \in \{LH, HH, \alpha L, XL\}$.
 	
	We next introduce the function spaces we will use. For the time domain, we first set $I_T = [T,\infty)$ for any $T > 0$. Let $p,q,s \in [1,\infty]$.
	
	Set $\cL^q(0,\infty) =  L^q((0,\infty), r^2 \dd r)$. We then define the space $\cL^q L^s(\R^3)$, anisotropic in the radial and the angular variable, by the norm
	\begin{align*}
		\|f\|_{\cL^q_r L^s_\theta} = \Big(\int_0^\infty \Big(\int_{S^2} |f(r \theta)|^s \dd \theta\Big)^{\frac{q}{s}} r^2 \dd r\Big)^{\frac{1}{q}}
	\end{align*}
	with the usual adaptions if $q$ or $s$ equals $\infty$.
	
	If $Z$ is a function space over $\R^3$ with norm $\|\cdot\|_{Z}$, we define the time-weighted spaces $L^q_\sigma(I_T,Z)$ by the norm
	\begin{align*}
		\|g\|_{L^q_\sigma Z} = \| t^\sigma g(t,x) \|_{L^q_t Z_x} = \|t^\sigma g(t,x)\|_{L^q Z(I_T \times \R^3)}
	\end{align*}		
	for any $\sigma \geq 0$. If we want to point out the underlying time interval, we also write $\|\cdot\|_{L^q_\sigma Z(I_T)}$. Note that if $\sigma = 0$ this is the standard $L^q(I_T,Z)$-space. In this case we simply write $\| \cdot \|_{L^q Z}$ or $\| \cdot \|_{L^q_t Z}$ if there is no chance of confusion with the time weight.
	
	 We use the standard homogeneous Besov spaces $\dot{B}^{\mu}_{q,2}(\R^3)$ as well as the Besov-type spaces $\dot{B}^{\mu}_{(q,s),2}(\R^3)$ and $\dot{B}^{\mu}_{p,(q,s),2}(I \times \R^3)$ defined by the norms
	\begin{align*}
		\|f\|_{\dot{B}^{\mu}_{q,2}} &= \Big(\sum_{k \in \Z} 2^{2k\mu} \|P_k f\|_{L^q}^2 \Big)^{\frac{1}{2}}\\
		\|f\|_{\dot{B}^{\mu}_{(q,s),2}} &= \Big(\sum_{k \in \Z} 2^{2k\mu} \|P_k f\|_{\cL^q_r L^s_\theta}^2 \Big)^{\frac{1}{2}}\\
		\|g\|_{\dot{B}^{\mu}_{p,(q,s),2}} &= \Big(\sum_{k \in \Z} 2^{2k\mu} \|P_k g\|_{L^p_t\cL^q_r L^s_\theta}^2 \Big)^{\frac{1}{2}}\\
	\end{align*}
	for $\mu \in \R$, where we only work on the time intervals $I = I_T$ or $I = \R$. Again, if we want to point out the underlying time interval, we write $\| \cdot \|_{\dot{B}^{\mu}_{p,(q,s),2}(I)}$.
	In the case $q = s$ we also write $\dot{B}^\mu_{p,q,2}(I \times \R^3) = \dot{B}^\mu_{p,(q,q),2}(I \times \R^3)$.
	
	We now fix $0 < \ep \ll 1$ for the rest of the paper and set
	\begin{equation}
	\label{eq:Defqeps}
		\frac{1}{q(\ep)} = \frac{1}{4} + \frac{\ep}{3} \qquad \text{and} \qquad \frac{1}{q(-\ep)} = \frac{1}{4} - \frac{\ep}{3}.
	\end{equation}
	We further fix $0 < \nu \ll 1$ and set $\sigma = \frac{1}{2} - \nu$.
	For the construction of the solution of the Zakharov system we will employ the spaces
	\begin{align*}
		X_T^\sigma(I_T \times \R^3) &= \{\nlschr \in C(I_T, H^1(\R^3)) \colon \| \nlschr \|_{X_T^\sigma} < \infty\} \\
		Y_T^\sigma(I_T \times \R^3) &= \{\nlwave \in C(I_T, L^2(\R^3)) \colon \| \nlwave \|_{Y_T^\sigma} < \infty\}
	\end{align*}
	 equipped with the norms
	\begin{align*}
		\| \nlschr \|_{\schrsp} &= \|\nlschr\|_{L^\infty_\sigma H^1} + \|\langle \nabla \rangle \nlschr\|_{L^2_\sigma \dot{B}^0_{6,2}} + \|\langle \nabla \rangle \nlschr\|_{L^2_\sigma \dot{B}^{\frac{1}{4}+\ep}_{(q(\ep), \frac{2}{1-\nu}),2}}, \\
		\| \nlwave \|_{\wavesp} &= \| \nlwave \|_{L^\infty_\sigma L^2}.
	\end{align*}

 	\subsection{Normal form reduction for the Zakharov system}
 	\label{ss:NormalForm}
 	
	We first note that in our proofs the term $\overline{v} u$ can be treated in the same way as $v u$ so that we drop the real part in~\eqref{eq:ZakharovSystemFirstOrder} for convenience. 	
 	
	The Zakharov system~\eqref{eq:ZakharovSystemFirstOrder} with final data $(u_+, v_+)$ can be rewritten as
	\begin{equation}
		\label{eq:ZakharovSystemFinalState}
		\begin{aligned}
			u(t) &= \eul^{\imu t \Delta} u_+ + \imu \int_t^\infty \eul^{\imu (t-s) \Delta} (v u)(s) \dd s, \\
			v(t) &= \eul^{\imu \alpha t |\nabla|} v_+ - \imu \alpha \int_t^\infty \eul^{\imu \alpha (t-s) |\nabla|}(|\nabla| |u|^2)(s) \dd s.
		\end{aligned}
	\end{equation}	 	
 	
 	In~\cite{GN2014} the authors introduced a normal form transform for the Zakharov system. For the Schr{\"o}dinger part it relies on the observation that the resonance function 
 	\begin{align*}
 		\resf(\xi-\eta,\eta) = |\xi|^2 + \alpha |\xi - \eta| - |\eta|^2
 	\end{align*}
 	does not vanish on the support of $\cP_{XL}$ as we have $\alpha \nsim |\xi - \eta| \sim |\xi| \gg |\eta|$ for these $\xi$ and $\eta$. This implies that~\eqref{eq:ZakharovSystemFinalState} is equivalent - at least for sufficiently smooth solutions - to
	\begin{equation}
		\label{eq:ZakharovSystemNormalForm}
		\begin{aligned}
			u(t) &= \eul^{\imu t \Delta} u_+ - \bdyop(v,u)(t) + \imu \int_t^\infty \eul^{\imu (t-s) \Delta} (v u)_{LH+HH+\alpha L}(s) \dd s \\
			&\qquad - \imu \alpha \int_t^\infty \eul^{\imu (t-s) \Delta} \bdyop(|\nabla| |u|^2,u)(s) \dd s + \imu \int_t^\infty \eul^{\imu (t-s) \Delta} \bdyop(v, vu)(s) \dd s, \\
			v(t) &= \eul^{\imu \alpha t |\nabla|} v_+ - \imu \alpha \int_t^\infty \eul^{\imu \alpha (t-s) |\nabla|}(|\nabla| |u|^2)(s) \dd s,
		\end{aligned}
	\end{equation}	 	
	where the bilinear Fourier multiplier $\bdyop$ is defined as
	\begin{align*}
		\bdyop(f,g) = \cF^{-1} \int_{\R^3}\frac{\cP_{XL}(\xi-\eta,\eta)}{\resf(\xi-\eta,\eta)} \hat{f}(\xi - \eta) \hat{g}(\eta) \dd \eta.
	\end{align*}
	We refer to Section~2 in~\cite{GN2014} for the details. Note that the boundary term $\bdyop(v,u)(T)$ arising from integration by parts vanishes in $H^1(\R^3)$ as $T \rightarrow \infty$ for $(u,v)$ as in Theorem~\ref{thm:MainResult}, see~Lemma~\ref{lem:BilinearMultiplierEstimate} below. We further remark that in~\cite{GN2014} a normal form transform is also applied to the wave part of the Zakharov system. Moreover, refinements of~\eqref{eq:ZakharovSystemNormalForm} are possible by identifying further nonresonant regimes, see e.g.~\cite{GNW2013}. Both is not necessary for our purposes so that we work with the simpler form~\eqref{eq:ZakharovSystemNormalForm}.

 	\subsection{Time-weighted Strichartz estimates}
 	
 	Strichartz estimates are an indispensable tool in order to estimate the nonlinearities. Since we are working with time weights, we also need time-weighted Strichartz estimates. However, these variants follow from the unweighted estimates as was observed in~\cite{NY2019} for $L^p(I_T,L^q(\R^3))$.
 	We follow the argument from~\cite{NY2019} to prove the following.
	\begin{lem}
		\label{lem:TimeWeightedStrichartzAbstract}
		Let $T \geq 1$, $q, \tilde{q} \in [2,\infty]$, and $\si \geq 0$. Assume that $X$ and $Y$ are Banach spaces such that
		\begin{align*}
			\Big\| \int_t^\infty \eul^{\imu (t-s) \Delta} f(s) \dd s \Big\|_{L^q(I_{T'}, X)} \lesssim \|f\|_{L^{\tilde{q}'}(I_{T'},Y)}
		\end{align*}
		for all $f \in L^{\tilde{q}'}(I_{T}, Y)$ and $T' \in I_T$. We then have
		\begin{align*}
			\Big\| t^\si \int_t^\infty \eul^{\imu (t-s) \Delta} f(s) \dd s \Big\|_{L^q(I_T, X)} \lesssim \|t^\si f(t) \|_{L^{\tilde{q}'}(I_T,Y)}
		\end{align*}
		for all $f \in L^{\tilde{q}'}_\si(I_T, Y)$.
	\end{lem} 	
	
	\begin{proof}
		In the case $\si = 0$ there is nothing to show so we assume $\si > 0$ in the following.
		
		 Take $p \in [1,\infty)$. We write $\one_{I_T}$ for the indicator function of $I_T$ and set $g(t)= (\si p)^{\frac{1}{p}} t^{\si - \frac{1}{p}}$ for $t \in I_T$. We then have
		\begin{equation}
			\label{eq:TimeWeightIntegral}
			t^{\si p} = T^{\si p} + \int_T^t g(\tau)^p \dd \tau = T^{\si p} + \| g(\tau) \one_{I_\tau}(t)\|_{L^p_\tau(I_T)}^p
		\end{equation}
		for all $t \in I_T$. Combining~\eqref{eq:TimeWeightIntegral} with Fubini's theorem, we obtain for any measurable function $h \colon I_T \rightarrow \R$
		\begin{align}
		\label{eq:TimeWeightFubini}
			\| t^\sigma h(t) \|_{L^p(I_T; \dd t)} &\leq \|T^\si h(t)\|_{L^p(I_T; \dd t)} + \|g(\tau) \one_{I_\tau}(t) h(t) \|_{L^p(I_T; \dd t) L^p(I_T; \dd \tau)} \nonumber\\
			&= T^{\si} \|h\|_{L^p(I_T)} + \|g(\tau) \|h(t)\|_{L^p(I_\tau; \dd t)} \|_{L^p(I_T; \dd \tau)}.
		\end{align}
		We now assume $q < \infty$ and employ~\eqref{eq:TimeWeightFubini} with $p = q$ and the assumption in order to infer
		\begin{align*}
			&\Big\| t^\si \int_t^\infty \eul^{\imu (t-s) \Delta} f(s) \dd s \Big\|_{L^q(I_T, X)} \leq T^\si \Big\|  \int_t^\infty \eul^{\imu (t-s) \Delta} f(s) \dd s \Big\|_{L^q(I_T, X)} \\
			&\hspace{10em}+ \Big\| g(\tau) \Big\| \int_t^\infty \eul^{\imu (t-s) \Delta} f(s) \dd s \Big\|_{L^q(I_\tau, X; \dd t)}\Big\|_{L^q(I_T; \dd \tau)} \\
			&\lesssim T^\si \|f\|_{L^{\tilde{q}'}(I_T,Y)} + \|g(\tau) \|f\|_{L^{\tilde{q}'}(I_\tau, Y)} \|_{L^q(I_T;\dd \tau)}.
		\end{align*}
		Using Minkowski's inequality and~\eqref{eq:TimeWeightIntegral}, we thus arrive at
		\begin{align*}
			 &\Big\| t^\si \int_t^\infty \eul^{\imu (t-s) \Delta} f(s) \dd s \Big\|_{L^q(I_T, X)}\lesssim T^\si \|f\|_{L^{\tilde{q}'}(I_T,Y)} + \|g(\tau) \|f\|_{L^{\tilde{q}'}(I_\tau, Y)} \|_{L^q(I_T;\dd \tau)}\\
			&\lesssim  \|t^\si f(t)\|_{L^{\tilde{q}'}(I_T,Y; \dd t)} + \| \|g(\tau) \one_{I_\tau}(t) \|f(t)\|_{Y} \|_{L^q(I_T;\dd \tau)} \|_{L^{\tilde{q}'}(I_T; \dd t)} \\
			&\lesssim  \|t^\si f(t)\|_{L^{\tilde{q}'}(I_T,Y; \dd t)}.
		\end{align*}
		In the case $q = \infty$, the assumption implies
		\begin{equation}
		\label{eq:TimeWeightqinfty}
			\Big\|  \int_t^\infty \eul^{\imu (t-s) \Delta} f(s) \dd s \Big\|_{X} \lesssim \|f\|_{L^{\tilde{q}'}(I_t,Y)} \lesssim t^{-\si} \Big(\int_t^\infty s^{\si \tilde{q}'} \|f(s)\|_{Y}^{\tilde{q}'} \dd s\Big)^{\frac{1}{\tilde{q}'}}
		\end{equation}
		for almost every $t \in I_T$ and the assertion follows.
	\end{proof}
 	
Choosing $X$ and $Y$ appropriately, we can extend known Strichartz respectively generalized Strichartz estimates to the time-weighted case. To that purpose we first recall that a pair $(q,r)$ is said to be \emph{Schr{\"o}dinger admissible} if
	\begin{align*}
		2 \leq q,r \leq \infty, \qquad \frac{2}{q} + \frac{3}{r}  = \frac{3}{2}.
	\end{align*}		
 	
 	We will also work with exponents from the extended range
	\begin{equation}
		\label{eq:ExtendedRangeAdmissiblePair}
		2 \leq q,r \leq \infty, \qquad \frac{1}{q} \leq \frac{5}{2}\Big(\frac{1}{2} - \frac{1}{r}\Big), \qquad (q,r) \neq \Big(2,\frac{10}{3}\Big)
	\end{equation}	 	
 	 	
 	We then obtain the following estimates.
 	\begin{prop}
 		\label{prop:TimeWeightedStrichartzEstimatesSchroedinger}
 		Let $T \geq 1$ and $\si \geq 0$.
 		\begin{enumerate}
 			\item \label{it:TimeWeightedStrichartzadmadm} Let $(q,r)$ and $(\tilde{q}, \tilde{r})$ be Schr{\"o}dinger admissible pairs. Then
 			\begin{align*}
 				\Big\| \int_t^\infty \eul^{\imu (t-s) \Delta} f(s) \dd s \Big\|_{L^q_\si \dot{B}^0_{r,2}} 
 				\lesssim \| f \|_{L^{\tilde{q}'}_\si \dot{B}^0_{\tilde{r}',2}}.
 			\end{align*}
 			\item \label{it:TimeWeightedStrichartzextadm} Let $(\tilde{q}, \tilde{r})$ be Schr{\"o}dinger admissible with $\tilde{q} \neq 2$. Then
 			\begin{align*}
 				\Big\| \int_t^\infty \eul^{\imu (t-s) \Delta} f(s) \dd s \Big\|_{L^2_\si \dot{B}^{\frac{1}{4}+\ep}_{(q(\ep),\frac{2}{1-\nu}),2}} 
 				\lesssim \| f \|_{L^{\tilde{q}'}_\si \dot{B}^0_{\tilde{r}',2}}.
 			\end{align*}
 			\item \label{it:TimeWeightedStrichartzextext} Let $(q,r)$ be Schr{\"o}dinger admissible and let $(\tilde{q}, \tilde{r})$ satisfy~\eqref{eq:ExtendedRangeAdmissiblePair} with $\tilde{q} \neq 2$. Then
 			\begin{align*}
 				\Big\| \int_t^\infty \eul^{\imu (t-s) \Delta} f(s) \dd s \Big\|_{L^q_\si \dot{B}^{0}_{r,2} \cap L^2_\si \dot{B}^{\frac{1}{4}+\ep}_{(q(\ep),\frac{2}{1-\nu}),2}} 
 				\lesssim \| f \|_{L^{\tilde{q}'}_\si \dot{B}^{\frac{3}{2}-\frac{2}{\tilde{q}}-\frac{3}{\tilde{r}}}_{(\tilde{r}',2),2}}.
 			\end{align*}
 		\end{enumerate}
 	\end{prop}
 	
 	\begin{proof}
 		Part~\ref{it:TimeWeightedStrichartzadmadm} follows from Lemma~\ref{lem:TimeWeightedStrichartzAbstract} and Strichartz estimates for admissible pairs, see~\cite{KT1998}.
 		
 		For~\ref{it:TimeWeightedStrichartzextadm} we first note that the triple $(2, q(\ep), \frac{2}{1-\nu})$ satisfies the assumptions of Corollary~2.9 in~\cite{G2016}, which yields
 		\begin{align*}
 			\|\eul^{\imu t \Delta} P_0 g \|_{L^2_t \cL^{q(\ep)}_r L_\theta^{\frac{2}{1-\nu}}} \lesssim \| P_0 g \|_{L^2}
 		\end{align*}
 		for all $g \in L^2(\R^3)$. Applying this estimate to $g = \int_{I_T} \eul^{-\imu s \Delta}  f(s) \dd s$ and using Strichartz estimates for admissible pairs, we get
 		\begin{align}
 		\label{eq:estStrextadm}
 			\Big\| \int_{I_T} \eul^{\imu (t-s) \Delta} P_0 f(s) \dd s \Big\|_{L^2_t \cL^{q(\ep)}_r L_\theta^{\frac{2}{1-\nu}}} \lesssim \Big\| \int_{I_T} \eul^{-\imu s \Delta} P_0 f(s) \dd s\Big\|_{L^2} 
 			\lesssim \|P_0 f\|_{L^{\tilde{q}'}(I_T, L^{\tilde{r}'})}.
 		\end{align}
 		Rescaling to frequencies of size $2^k$, taking the $l^2$-norm in $k$, and employing the Christ-Kiselev lemma (see~\cite{CK2001, T2000}), we thus obtain the assertion.
 		
 		 Replacing the application of Strichartz estimates for admissible pairs in~\eqref{eq:estStrextadm} by the dual estimate one obtains from Theorem~1.1 in~\cite{G2016}, cf. \cite[Corollary~2.11]{G2016}, we infer part~\ref{it:TimeWeightedStrichartzextext} in the same way.
 	\end{proof}
 	
 	\begin{rem}
 	\label{rem:TimeWeightedEnergyEstimateWave}
 	The arguments from Lemma~\ref{lem:TimeWeightedStrichartzAbstract} and Proposition~\ref{prop:TimeWeightedStrichartzEstimatesSchroedinger} also work for the half-wave group. However, for the half-wave group we only use the energy estimate, where the time-weighted version
 		\begin{equation}
 			\label{eq:TimeWeightedEnergyEstimateWave}
 			\Big\|\int_t^\infty \eul^{\imu \alpha (t-s) |\nabla|} f(s) \dd s \Big\|_{L^\infty_\si L^2}
 			\lesssim \| f \|_{L^1_\si L^2}
 		\end{equation}
 		for all $f \in L^1_\si(I_T, L^2(\R^3))$ is immediate, cf.~\eqref{eq:TimeWeightqinfty}.
 	\end{rem}
 	
	In view of the large deviation estimate in Lemma~\ref{lem:LargeDeviationEstimate} below, we will not be able to use $L^\infty$-norms in time for the linear solutions. We will avoid these $L^\infty$-norms by employing the following Sobolev embedding type estimates.
	 	
 	\begin{lem}
 		\label{lem:SobolevSchroedingerWave}
 		Fix $T \geq 1$,  $\mu \in \R$, and exponents $q \in [1,\infty)$ and $r \in (2,\infty)$. 
 		\begin{enumerate}
 			\item \label{it:SobolevSchroedinger} Take $u_0 \in L^2(\R^3)$ such that $\linschr(t) = \eul^{\imu t \Delta} u_0$ belongs to $L^q_\si(I_T, \dot{B}^{\mu + \frac{2}{q}}_{r,2}(\R^3))$. Then
 				\begin{align*}
 					\|P_k \linschr \|_{L^\infty_\si L^r} \lesssim 2^{\frac{2}{q} k}\|P_k \linschr \|_{L^q_\si L^r}
 				\end{align*}
 				for all $k \in \Z$.
 			\item \label{it:SobolevWave} Take $ v_0 \in L^2(\R^3)$ such that $\linwave(t) = \eul^{\imu \alpha t |\nabla|} v_0$ belongs to $\btwt{\mu + \frac{1}{q}}{q}{r}(I_T \times \R^3)$. Then
 				\begin{align*}
 					\|\linwave\|_{L^\infty \dot{B}^\mu_{r,2}} \lesssim \|\linwave\|_{\dot{B}^{\mu + \frac{1}{q}}_{q,r,2}}.
 				\end{align*}
 		\end{enumerate}
 	\end{lem}
 	
 	\begin{proof}
 		\ref{it:SobolevSchroedinger} Let $k \in \Z$. Since $P_k \linschr$ is smooth and $\|P_k \linschr(t)\|_{L^r} \rightarrow 0$ as $t \rightarrow \infty$, we compute
 		\begin{align*}
 			&\| P_k \linschr(t) \|_{L^r}^q = - \int_t^\infty \partial_t \|P_k \linschr \|_{L^r}^q \dd s \\
 			&= -q \int_{t}^\infty \|P_k \linschr(s)\|_{L^r}^{q-r} \int_{\R^3} |P_k \linschr(s)|^{r-2} \Re(\overline{P_k \linschr} \partial_t P_k \linschr)(s) \dd x \dd s \\
 			&\lesssim t^{-q \si} \int_{t}^\infty s^{q \si} \|P_k \linschr(s)\|_{L^r}^{q-1} \|\partial_t  P_k \linschr(s)\|_{L^r} \dd s.
 		\end{align*}
 		Using that $\linschr$ is a solution of the linear Schr{\"o}dinger equation and thus 
			\begin{align*}
				\|\partial_t P_k \linschr \|_{L^r} = \| \Delta P_k \linschr\|_{L^r} \lesssim 2^{2k} \|P_k \linschr\|_{L^r},
			\end{align*}			 		
 		we infer
 		\begin{align*}
 			t^{\si} \| P_k \linschr(t) \|_{L^r} \lesssim 2^{\frac{2}{q} k} \|P_k \linschr\|_{L^q_\si L^r}.
 		\end{align*}
 		Taking the $L^\infty$-norm in time, the assertion now follows.
 		
 		\ref{it:SobolevWave} As in~\ref{it:SobolevSchroedinger} we derive 
 		\begin{align*}
 			\|P_k \linwave(t)\|_{L^r} \lesssim 2^{\frac{1}{q}k} \|P_k \linwave \|_{L^q L^r}
		\end{align*}
		 for all $k \in \Z$. Multiplying with $2^{\mu k}$, taking the $l^2$-norm in $k$ and then the $L^\infty$-norm in $t$, we obtain the assertion.
 	\end{proof}
 
 \section{Linear estimates with randomized data}
 \label{sec:LinearRandomizedEstimates}
 
	In this section we prove improved space-time estimates for the linear solutions of the Schr{\"o}dinger and half-wave equation with physical-space respectively angular randomized data. 
 
 	\subsection{Probabilistic estimates}
 	We start by collecting two probabilistic results. The first one is a crucial tool in the proof of randomization improved estimates.
 	This large deviation estimate can be found in Lemma~3.1 in~\cite{BT2008I}.
 	\begin{lem}
 		\label{lem:LargeDeviationEstimate}
 		Let $(X_k)_{k \in \N}$ be a sequence of independent, real-valued, zero-mean random variables on a probability space $(\Omega, \cA, \PP)$ whose distributions $(\mu_k)_{k \in \N}$ satisfy~\eqref{eq:ConditionRandomVariables}. Then there is a constant $C > 0$ such that
 		\begin{align*}
 			\Big\| \sum_{k \in \N} c_k X_k \Big\|_{L^\beta(\Omega)} \leq C \sqrt{\beta} \, \Big( \sum_{k \in \N} |c_k|^2 \Big)^{\frac{1}{2}}
 		\end{align*}
 		for all $\beta \in [2,\infty)$ and $(c_k)_{k \in \N} \in l^2(\N)$.
 	\end{lem}
 	
 	The next lemma is taken from~\cite[Lemma~2.4]{M2019}, see~\cite[Lemma~4.5]{T2010} or \cite[Lemma~2.2]{BOPII2015} for a proof.
 	\begin{lem}
 		\label{lem:MeasureFromLargeDeviation}
 		Let $F$ be a measurable function on a probability space $(\Omega, \cA, \PP)$. Assume that there are constants $C, A > 0$ and $\beta_0 \geq 1$ such that
 		\begin{align*}
 			\|F\|_{L^\beta(\Omega)} \leq C \sqrt{\beta} A
 		\end{align*}
 		for all $\beta \geq \beta_0$. Then there are constants $C', c > 0$ such that 
 		\begin{align*}
 			\PP(\omega \in \Omega \colon |F(\omega)| > \lambda) \leq C' \eul^{- c \lambda^2 A^{-2}}
 		\end{align*}
 		for all $\lambda > 0$.
 	\end{lem}
 
 	\subsection{The linear Schr{\"o}dinger equation with physical-space randomized data}
 	
 	We first show improved space-time estimates for the solution of the linear Schr{\"o}dinger equation with data randomized in physical space as in~\eqref{eq:DefPhysicalSpaceRandomization}. In a first step we note that the Littlewood-Paley operators do not commute with the physical-space randomization. We will thus need the following estimates. Similar ones have already appeared in~\cite{B2019,DLM2019}.
	\begin{lem}
		\label{lem:MismatchEstimates}
		Let $1 \leq p < \infty$, $k, j \in \N$ with $|k - j| \geq 5$, $l, l' \in \Z^3$,  $D > 0$, and $(\psi_m)_{m \in \Z^3}$ the partition of unity introduced in~\eqref{eq:DefPartitionUnityPhysicalSpaceRandomization}. We then have
 	\begin{align}
 		\label{eq:MismatchSpaceFrequencySpace}
 		&\| \psi_l P_k (\psi_{l'}f) \|_{L^{p}} + \| \psi_l P_{\leq 0} (\psi_{l'}f) \|_{L^{p}} \lesssim_D  \langle l - l' \rangle^{-D} \|f\|_{L^p}, \\
 		\label{eq:MismatchFrequencySpaceFrequency}
 		&\| P_k (\psi_l P_j f)\|_{L^{p}} \lesssim_D 2^{- D k} 2^{- D j} \|f\|_{L^p}, \\
 		\label{eq:MismatchFrequencySpaceFrequency2}
 		& \| P_k (\psi_l P_{\leq k - 5} f) \|_{L^{p} } \lesssim_D 2^{- D k} \|f\|_{L^p}
 	\end{align}
 	for all $f \in L^p(\R^3)$, where the implicit constants are independent of $j,k,l,$ and $l'$.
	\end{lem} 	
	
	The proof of Lemma~\ref{lem:MismatchEstimates} is straightforward. We provide the details in Appendix~\ref{app:ProofOfMismatch} for the convenience of the reader.

 	The previous lemma now implies the following result which fits to our applications.
	\begin{cor}
		\label{cor:EstimatePkpsill2kl2lLp}
		Let $p \in [2,\infty)$, $f \in H^1(\R^3)$, and $(\psi_l)_{l \in \Z^3}$ the partition of unity introduced in~\eqref{eq:DefPartitionUnityPhysicalSpaceRandomization}. We then have the estimate
		\begin{align*}
			\| (1+2^{2k})^{\frac{1}{2}} P_k (\psi_l f) \|_{l^2_k l^2_l L^{p'}_x} \lesssim \|f\|_{H^1}.
		\end{align*}
	\end{cor} 	
	\begin{proof}
		We first note that in the case $k \leq 4$ we get
		\begin{align}
			\|P_k (\psi_l f)\|_{l^2_l L^{p'}_x} &\lesssim 2^{3k(1-\frac{1}{p'})} \|P_k (\psi_l f)\|_{l^2_l L^1_x} 
			\lesssim 2^{\frac{3}{p}k} \| \psi_l f \|_{l^2_l L^1_x} \lesssim 2^{\frac{3}{p}k} \| \psi_l f \|_{l^2_l L^2_x} \nonumber\\
			&\lesssim 2^{\frac{3}{p}k} \| f \|_{L^2_x}, \label{eq:EstimatePkpsilfsmallk}
		\end{align}
		where we used Bernstein's and H{\"o}lder's inequality as well as $|\supp \psi_l| \lesssim 1$ for all $l \in \Z^3$.
		Next we fix an integer $k \geq 5$. We first note that
		\begin{align}
		\label{eq:SplittingPkpsilf}
			\|P_k (\psi_l f)\|_{l^2_l L^{p'}_x} \leq \|P_k (\psi_l \tilde{P}_k f)\|_{l^2_l L^{p'}_x} 
					+ \!\! \sum_{\substack{j \in \N \\ |j - k| \geq 5}} \|P_k (\psi_l P_j f)\|_{l^2_l L^{p'}_x} 
					+ \|P_k (\psi_l P_{\leq 0} f)\|_{l^2_l L^{p'}_x}
		\end{align}
		For the first term on the right-hand side we simply apply H{\"o}lder's inequality as above to deduce
		\begin{align*}
			\|P_k (\psi_l \tilde{P}_k f)\|_{l^2_l L^{p'}_x} \lesssim  \|\psi_l \tilde{P}_k f\|_{l^2_l L^{2}_x} \lesssim \|\tilde{P}_k f\|_{L^{2}_x}.
		\end{align*}
		For $j \in \N$ with $|j - k| \geq 5$ we employ Lemma~\ref{lem:MismatchEstimates} to derive
		\begin{align*}
			&\|P_k (\psi_l P_j f)\|_{L^{p'}_x} \leq \sum_{l' \in \Z^3} \|P_k (\psi_l P_j (\psi_{l'}f))\|_{L^{p'}_x} \\
			&\lesssim \sum_{l' \in \Z^3} \| P_k (\psi_l P_j (\psi_{l'} f))\|_{L^{p'}_x }^{\frac{1}{2}} \| \psi_l P_j (\psi_{l'}  f) \|_{L^{p'}_x}^{\frac{1}{2}}
			\lesssim \sum_{l' \in \Z^3} 2^{-2k} 2^{-2j} \langle l - l' \rangle^{-20} \|\tilde{\psi}_{l'} f\|_{L^{p'}_x},
		\end{align*}
		for all $l \in \Z^3$, where $\tilde{\psi}_{l'} = \sum_{m \in \Z^3, |m - l'| \leq 4} \psi_m$ equals $1$ on the support of $\psi_l$. Using $\langle x - l \rangle \lesssim \langle l' - l \rangle$
		for all $x \in \supp \tilde{\psi}_{l'}$, H{\"o}lder's inequality, and $|\supp \tilde{\psi}_{l'}| \lesssim 1$ for all $l' \in \Z^3$, we thus obtain
		\begin{align*}
			&\|P_k (\psi_l P_j f)\|_{L^{p'}_x} \lesssim \sum_{l' \in \Z^3} 2^{-2k} 2^{-2j} \langle l - l' \rangle^{-20} \|\tilde{\psi}_{l'} f\|_{L^{2}_x} \\
			&\lesssim  \sum_{l' \in \Z^3} 2^{-2k} 2^{-2j} \langle l - l' \rangle^{-10} \|\langle x - l \rangle^{-10}  f\|_{L^{2}_x} 
			\lesssim 2^{-2k} 2^{-2j}  \|\langle x - l \rangle^{-10}  f\|_{L^{2}_x} 
		\end{align*}
		for all $l \in \Z^3$. Taking the $l^2$-norm in $l$ and summing over $j$, we arrive at
		\begin{align*}
			\sum_{\substack{j \in \N \\ |j - k| \geq 5}} \|P_k (\psi_l P_j f)\|_{l^2_l L^{p'}_x} \lesssim 2^{-2k} \|f\|_{L^2_x}.
		\end{align*}
		The same arguments also yield
		\begin{align*}
			\|P_k (\psi_l P_{\leq 0} f)\|_{l^2_l L^{p'}_x} \lesssim 2^{-2k} \|f\|_{L^2_x}.
		\end{align*}
		In view of~\eqref{eq:SplittingPkpsilf}, we have thus shown
		\begin{align}
		\label{eq:EstimatePkpsilflargek}
			\|P_k (\psi_l f)\|_{l^2_l L^{p'}_x} \lesssim \| \tilde{P}_k f\|_{L^2_x} + 2^{-2k} \|f\|_{L^2_x}
		\end{align}
		for every $k \geq 5$. Multiplying~\eqref{eq:EstimatePkpsilfsmallk} and~\eqref{eq:EstimatePkpsilflargek} with $(1+2^{2k})^{\frac{1}{2}}$ and taking the $l^2$-norm in $k$ over $k \leq 4$ respectively $k \geq 5$, the assertion now follows.
	\end{proof}
 	
 	We finally turn our attention to the randomization improved space-time estimate. The key point is that the randomization in physical space allows us to apply the dispersive estimate although the data only belongs to $L^2$.
 	\begin{prop}
 		\label{prop:RandomizationImprovedEstimatesSchroedinger}
 		Fix $T \geq 1$. Take $q,r \in [2,\infty)$ and $\mu > 0$ such that
 		\begin{align*}
 			\frac{3}{2} - \frac{1}{q} - \frac{3}{r} - \mu > 0.
 		\end{align*}
 		Take $u_+ \in H^1(\R^3)$ and let $u_+^\om$ be its physical-space randomization from~\eqref{eq:DefPhysicalSpaceRandomization}. Then there is a constant $C > 0$ such that
 		\begin{align*}
 			\| \langle \nabla \rangle \eul^{\imu t \Delta} u_+^\om \|_{L_\om^\beta L^q_{\mu} \dot{B}^0_{r,2}(I_T)} \leq C \sqrt{\beta}\, T^{-\frac{3}{2} + \frac{1}{q} + \frac{3}{r} + \mu} \|u_+\|_{H^1}
 		\end{align*}
 		for all $\beta \in [1,\infty)$.
 	\end{prop}
 	
 	\begin{proof}
 		Since $\Omega$ is a probability space, it is enough to prove the assertion for all $\beta \in [\max\{q,r\},\infty)$.
 		Using Minkowski's inequality, Lemma~\ref{lem:LargeDeviationEstimate}, and the dispersive estimate, we deduce
 		\begin{align*}
 			&\| \langle \nabla \rangle \eul^{\imu t \Delta} u_+^\om \|_{L_\om^\beta L^q_{\mu} \dot{B}^0_{r,2}} \lesssim \Big\| \Big( \sum_{k \in \Z} (1 + 2^{2k}) \|\eul^{\imu t \Delta} P_k u_+^\om\|_{L^r_x}^2 \Big)^{\frac{1}{2}} \Big\|_{L^\beta_\om L^q_\mu} \\
 			&\lesssim  \Big\| (1 + 2^{2k})^{\frac{1}{2}} \Big\| \sum_{l \in \Z^3} X_l \eul^{\imu t \Delta} P_k (\psi_l u_+) \Big\|_{L^\beta_\om} \Big\|_{l^2_k L^q_\mu L^r_x} \\
 			&\lesssim \sqrt{\beta}  \Big\| (1 + 2^{2k})^{\frac{1}{2}}  \|  \eul^{\imu t \Delta} P_k (\psi_l u_+) \|_{l^2_l} \Big\|_{l^2_k L^q_\mu L^r_x} \\
 			&\lesssim \sqrt{\beta}  \Big\| (1 + 2^{2k})^{\frac{1}{2}}  \|  \eul^{\imu t \Delta} P_k (\psi_l u_+) \|_{L^q_\mu L^r_x} \Big\|_{l^2_k l^2_l} \\
 			&\lesssim \sqrt{\beta}  \Big\| (1 + 2^{2k})^{\frac{1}{2}}  \| t^\mu t^{-\frac{3}{2} + \frac{3}{r}} \|P_k(\psi_l u_+)\|_{L^{r'}_x} \|_{L^q(T,\infty)} \Big\|_{l^2_k l^2_l} \\
 			&\lesssim \sqrt{\beta} \, T^{-\frac{3}{2} + \frac{1}{q} + \frac{3}{r} + \mu} \| (1 + 2^{2k})^{\frac{1}{2}}   \|P_k(\psi_l u_+)\|_{l^2_k l^2_l L^{r'}_x} 
 			\lesssim \sqrt{\beta} \, T^{-\frac{3}{2} + \frac{1}{q} + \frac{3}{r} + \mu} \|u_+\|_{H^1},
 		\end{align*}
 		where we applied Corollary~\ref{cor:EstimatePkpsill2kl2lLp} in the last line.
 	\end{proof}
 	
 	\subsection{The linear half-wave equation with angular randomized data}	
 	We next show the improved space-time estimates for the half-wave group with data randomized in the angular variable as in~\eqref{eq:DefAngularRandomization}. We combine the proof of~\cite[Proposition~1.2]{BK2019} with the additional observation that this angular randomization does not only extend the range of space-time exponents $(p,q)$, but that we also obtain arbitrary high integrability in the angular variable.
 	\begin{prop}
 		\label{prop:RandomizationImprovedEstimatesHalfWave}
 		Take $p,q,s \in [2,\infty)$ such that $(p,q)$ satisfies
 		\begin{align*}
 			\frac{1}{p} + \frac{2}{q} < 1.
 		\end{align*}
 		Take $v_+ \in L^2(\R^3)$ and let $v_+^\omega$ be its angular randomization from~\eqref{eq:DefAngularRandomization}. Then there is a constant $C > 0$ such that
 		\begin{align*}
 			\| \eul^{\imu \alpha t |\nabla|} v_+^\omega \|_{L^\beta_\omega \dot{B}^{\frac{1}{p}+\frac{3}{q}-\frac{3}{2}}_{p,(q,s),2}} \leq C \sqrt{\beta} \|v_+\|_{L^2}
 		\end{align*}
 		for all $\beta \in [1,\infty)$.
 	\end{prop}
 	
 	\begin{proof}
 	Using that $\Omega$ is a probability space, we only have to show the assertion for $\beta \in [\max\{p,q,s\},\infty)$.
 	
 		The constant $\alpha$ is irrelevant for the proof so that we set $\alpha = 1$ to ease the notation a bit.
 		We again write $g_m = (P_m v_+)(2^{-m} \cdot)$ for the rescaled Littlewood-Paley blocks of $v_+$ used in the definition of the angular randomization.
 		
 		By scaling and an application of Minkowski's inequality it is enough to show
 		\begin{equation}
 		\label{eq:EstimateAngularRandomizationScalingReduced}
 			\| \eul^{\imu t |\nabla|} g_m^\omega \|_{L^\beta_\omega L^p_t \cL_r^q L^s_\theta} \leq C \sqrt{\beta} \|g_m\|_{L^2_x}
 		\end{equation}
 		for all $m \in \Z$, where $g_m^\omega$ is defined in~\eqref{eq:gmInGoodFrame}. Fix $m \in \Z$.
 		
 		In view of~\eqref{eq:gmhatInGoodFrame}, \eqref{eq:gmInGoodFrame}, and~\eqref{eq:Randomizationgm}, we have the representation
 		\begin{align*}
 			\eul^{\imu t |\nabla|} g_m^\omega(r \theta) = \sum_{k,l} a_k r^{-\frac{1}{2}} Y^m_{k,l}(\omega) b_{k,l}(\theta) \int_0^\infty \eul^{\imu t \rho }\hat{c}^m_{k,l}(\rho) J_{\frac{2k+1}{2}}(r \rho) \rho^{\frac{3}{2}} \dd \rho
 		\end{align*}
 		with $r \in (0,\infty)$ and $\theta \in S^2$. Here and in the following we write $\sum_{k,l}$ for the sum over $l \in \{1, \ldots, N_k\}$ and $k \in \N_0$ as well as $\|\cdot\|_{l^2_{k,l}}$ for the corresponding $l^2$-norm.
 		
 		Using Minkowski's inequality, Lemma~\ref{lem:LargeDeviationEstimate}, and property~\eqref{eq:LqBoundGoodFrame} of the good frame, we obtain
 		\begin{align}
 			\| \eul^{\imu t |\nabla|} g_m^\omega \|_{L^\beta_\omega L^p_t \cL_r^q L^s_\theta}
 			&\lesssim \sqrt{\beta} \, \Big\| r^{-\frac{1}{2}} b_{k,l}(\theta) \int_0^\infty \eul^{\imu t \rho } \hat{c}^m_{k,l}(\rho) J_{\frac{2k+1}{2}}(r \rho) \rho^{\frac{3}{2}} \dd \rho \Big\|_{L^p_t \cL_r^q L^s_\theta l^2_{k,l}} \nonumber\\
 			&\lesssim \sqrt{\beta} \, \Big\| r^{-\frac{1}{2}} \int_0^\infty \eul^{\imu t \rho } \hat{c}^m_{k,l}(\rho) J_{\frac{2k+1}{2}}(r \rho) \rho^{\frac{3}{2}} \dd \rho \Big\|_{ l^2_{k,l} L^p_t \cL_r^q}, \label{eq:EstimateLargeDeviationGoodFrame}
 		\end{align}
 		where the implicit constant is independent of $\beta$. Next we recall that $g_m$ has unit frequency so that $\hat{c}^m_{k,l}$ is supported in $\{\rho \in (0,\infty) \colon \frac{1}{2} < \rho < 2\}$ for all $l,k$. We can thus write $\hat{c}^m_{k,l}$ as a Fourier series in $L^2(0,4)$, which yields
 		\begin{equation}
 			\label{eq:chatmklAsFourierSeries}
 			\hat{c}^m_{k,l}(\rho) = \sum_{n \in \Z} c^{m,n}_{k,l} \eul^{\imu \frac{\pi}{2} n \rho}.
 		\end{equation}
 		Consequently, we get
 		\begin{align*}
 			r^{-\frac{1}{2}} \int_0^\infty \eul^{\imu t \rho } \hat{c}^m_{k,l}(\rho) J_{\frac{2k+1}{2}}(r \rho) \rho^{\frac{3}{2}} \dd \rho
 			= \sum_{n \in \Z}  c^{m,n}_{k,l} r^{-\frac{1}{2}} \psi_{t + \frac{\pi}{2}n}^k(r)
 		\end{align*}
 		with
 		\begin{align*}
 			\psi^k_{t + \frac{\pi}{2}n}(r) = \int_0^\infty \eul^{\imu (t + \frac{\pi}{2}n) \rho }   J_{\frac{2k+1}{2}}(r \rho) \chi(\rho) \dd \rho,
 		\end{align*}
 		where we absorbed $\rho^{\frac{3}{2}}$ in the smooth bump function $\chi$ which is supported in $(0,4)$. The asymptotic properties of the functions $\psi_{t+\frac{\pi}{2}n}^k$ derived in Proposition~4.1 in~\cite{S2005} yield
 		\begin{align*}
 			&\Big|\sum_{n \in \Z}  r^{-\frac{1}{2}} \psi_{t + \frac{\pi}{2}n}^k(r)\Big| \\
 			&\lesssim \sum_{n \in \Z} \frac{1}{(1+|t + \frac{\pi}{2} n|)(1+|r-|t + \frac{\pi}{2}n||)^{\frac{1}{2}}} \! \left[\frac{1}{(1+|r-|t + \frac{\pi}{2}n||)^{\frac{1}{2}}} + R(k, t + \tfrac{\pi}{2}n,r) \right]
 		\end{align*}
 		with
 		\begin{align*}
 			\sum_{n \in \Z} \frac{1}{1+|r-|t + \frac{\pi}{2}n||} R^2(k, t + \tfrac{\pi}{2}n,r) \lesssim 1,
 		\end{align*}
 		where the implicit constants are independent of $l,k$ and $r$. See~(5.11), (5.12), and~(5.13) in~\cite{S2005} for the details.
 		
 		Applying the Cauchy-Schwarz inequality, we thus get
 		\begin{align*}
 			\Big\| \sum_{n \in \Z} c^{m,n}_{k,l} r^{-\frac{1}{2}} \psi_{t + \frac{\pi}{2}n}^k(r) \Big\|_{L^\infty_r} \lesssim \left(\sum_{n \in \Z} \frac{|c^{m,n}_{k,l}|^2}{(1+|t + \frac{\pi}{2} n|)^2} \right)^{\frac{1}{2}}
 		\end{align*}
 		and therefore
 		\begin{equation}
 			\label{eq:EstimateLinftyInr}
 			\Big\| r^{-\frac{1}{2}} \int_0^\infty \eul^{\imu t \rho } \hat{c}^m_{k,l}(\rho) J_{\frac{2k+1}{2}}(r \rho) \rho^{\frac{3}{2}} \dd \rho \Big\|_{L^\infty_r} \lesssim \left(\sum_{n \in \Z} \frac{|c^{m,n}_{k,l}|^2}{(1+|t + \frac{\pi}{2} n|)^2} \right)^{\frac{1}{2}}.
 		\end{equation}
 		On the other hand, the energy estimate implies
 		\begin{align}
 			\label{eq:EstimateL2Inr}
 			&\Big\| r^{-\frac{1}{2}} \int_0^\infty \eul^{\imu t \rho } \hat{c}^m_{k,l}(\rho) J_{\frac{2k+1}{2}}(r \rho) \rho^{\frac{3}{2}} \dd \rho \Big\|_{\cL^2_r} \nonumber\\
 			&= \Big\| b_{k,l}(\theta) r^{-\frac{1}{2}} \int_0^\infty \eul^{\imu t \rho } \hat{c}^m_{k,l}(\rho) J_{\frac{2k+1}{2}}(r \rho) \rho^{\frac{3}{2}} \dd \rho \Big\|_{L^2_x} \lesssim \|\hat{c}^m_{k,l}(\rho) b_{k,l}(\theta)\|_{L^2_\xi} \nonumber \\
 			&\lesssim \|\hat{c}^m_{k,l}(\rho)\|_{\cL^2_{\rho}} \lesssim  \|\hat{c}^m_{k,l}(\rho)\|_{L^2_{\rho}} \lesssim \Big(\sum_{n \in \Z} |c^{m,n}_{k,l}|^2 \Big)^{\frac{1}{2}},
 		\end{align}
 		where we again used that $\supp \hat{c}^m_{k,l} \subseteq \{\rho \in (0,\infty) \colon \frac{1}{2} < \rho < 2\}$ and Plancherel's theorem.
 		
 		Interpolating between~\eqref{eq:EstimateLinftyInr} and~\eqref{eq:EstimateL2Inr} for fixed time $t$ (see e.g.~\cite{BL1976}), we infer
 		\begin{align*}
 			&\Big\| r^{-\frac{1}{2}} \int_0^\infty \eul^{\imu t \rho } \hat{c}^m_{k,l}(\rho) J_{\frac{2k+1}{2}}(r \rho) \rho^{\frac{3}{2}} \dd \rho \Big\|_{\cL^q_r} \lesssim \left\| \frac{|c^{m,n}_{k,l}|}{(1+|t + \frac{\pi}{2} n|)^{(1 - \frac{2}{q})}} \right\|_{l^2_n}.
 		\end{align*}
 		Hence, using Minkowski's inequality and the assumption $\frac{1}{p} + \frac{2}{q} < 1$, we get
 		\begin{align*}
 			&\Big\| r^{-\frac{1}{2}} \int_0^\infty \eul^{\imu t \rho } \hat{c}^m_{k,l}(\rho) J_{\frac{2k+1}{2}}(r \rho) \rho^{\frac{3}{2}} \dd \rho \Big\|_{L^p_t \cL^q_r} \lesssim \| c^{m,n}_{k,l} \|_{l^2_n} 
 			\lesssim \| \hat{c}^m_{k,l} \|_{L^2} \lesssim \| \hat{c}^m_{k,l} \|_{\cL^2}.
 		\end{align*}
	Inserting this estimate into~\eqref{eq:EstimateLargeDeviationGoodFrame} and employing~\eqref{eq:L2Normgm}, we finally conclude
	\begin{align*}
		\| \eul^{\imu t |\nabla|} g_m^\omega \|_{L^\beta_\omega L^p_t \cL_r^q L^r_\theta}
 			&\lesssim \sqrt{\beta} \, \| \hat{c}^m_{k,l} \|_{l^2_{k,l} \cL^2} \lesssim \sqrt{\beta} \, \|g_m\|_{L^2_x},
	\end{align*}
	implying the assertion.
 	\end{proof}
 
 \section{Multilinear estimates for the Zakharov system}
 \label{sec:MultilinearEstimates}
 
	We now prove the multilinear estimates for the nonlinear terms appearing in~\eqref{eq:ZakharovSystemNormalForm}. In view of the application of these estimates in the proof of Theorem~\ref{thm:MainResult}, we decompose $u = \linschr + \nlschr$ and $v = \linwave + \nlwave$, where $\linschr$ and $\linwave$ are linear solutions of the Schr{\"o}dinger respectively half-wave equation. Note that $\linschr$ and $\linwave$ will be linear solutions with randomized data satisfying the improved bounds from Section~\ref{sec:LinearRandomizedEstimates} when we apply these estimates. In particular, we use different norms for $\linschr$ and $\nlschr$ respectively $\linwave$ and $\nlwave$ so that the different interactions often have to be treated separately.

 \subsection{Boundary terms}
 
In order to estimate the boundary terms $\bdyop$, we introduce the bilinear operators
\begin{equation}
	\label{eq:DefTm}
	T_m(f,g)(x) = \int_{\R^6} m(\xi,\eta) \hat{f}(\xi) \hat{g}(\eta) \eul^{\imu x (\xi + \eta)} \dd \xi \dd \eta \qquad (x \in \R^3)
\end{equation}
for $m \in L^{\infty}(\R^6)$ and $f,g \in \Schw(\R^3)$.

The following Coifman-Meyer-type bilinear multiplier estimate was proven in~\cite[Lemma~3.5]{GN2014}.

\begin{lem}
	\label{lem:BilinearMultiplierEstimate}
	Let  $m \in C^{\infty}(\R^{6})$ be bounded and assume that there are constants $C_{\alpha,\beta} > 0$ such that
	\begin{equation*}
		|\partial^\alpha_\xi \partial^\beta_\eta m(\xi,\eta)| \leq C_{\alpha,\beta} |\xi|^{-|\alpha|} |\eta|^{-|\beta|}
	\end{equation*}
	for all $\xi,\eta \in \R^3$ and $\alpha,\beta \in \N_0^3$. Take $p,q,r \in [1,\infty]$ with $\frac{1}{r} = \frac{1}{p} + \frac{1}{q}$.
	
	Then there is a constant $C > 0$ such that
	\begin{equation*}
		\|T_m(P_{k_1} f, P_{k_2} g) \|_{L^r} \leq C \|P_{k_1} f\|_{L^p} \|P_{k_2} g\|_{L^q}
	\end{equation*}
	for all $f \in L^p(\R^3)$, $g \in L^q(\R^3)$, and $k_1, k_2 \in \Z$, where $T_m$ is the operator defined in~\eqref{eq:DefTm}.
\end{lem}

We also use $\|\langle \nabla \rangle P_k f\|_{\cL_r^q L^s_\theta} \sim (1+2^{2k})^{\frac{1}{2}} \|P_k f\|_{\cL^q_r L^s_\theta}$ without further reference in the following, see~\cite{GLNW2014}.

Lemma~\ref{lem:BilinearMultiplierEstimate} can be applied to the bilinear operator $\bdyop$. It shows that, roughly speaking, $\bdyop$ acts like 
\begin{align*}
	 \bdyop(f,g) \sim |\nabla|^{-1} \langle \nabla \rangle^{-1} (fg)_{XL}.
\end{align*}

We obtain the following estimates for the boundary terms.
\begin{lem}
	\label{lem:EstimateBoundaryTerms}
	Fix $T \geq 1$ and let $0 < \eta \leq r$. Take $u_0, v_0 \in L^2(\R^3)$ with $\|u_0\|_{L^2} + \|v_0\|_{L^2} \leq r$ such that $\linschr(t) = \eul^{\imu t \Delta} u_0$ and $\linwave(t) = \eul^{\imu \alpha t |\nabla|} v_0$ satisfy
	\begin{align*}
		&\| \langle \nabla \rangle \linschr \|_{L^{\frac{2}{\nu}}_\si \dot{B}^0_{3,2}(I_T)} + \| \langle \nabla \rangle \linschr \|_{L^{2}_\si \dot{B}^0_{6,2}(I_T)} + \|\linwave\|_{\dot{B}^{-\frac{1}{2}}_{4,4,2}(I_T)} + \|\linwave\|_{\dot{B}^{-\frac{1}{2} + \frac{4}{3}\ep}_{q(\ep),q(\ep),2}(I_T)} \leq \eta.
	\end{align*}
	We then have the estimates
	\begin{align}
		\| \bdyop(\linwave, \linschr)\|_{\schrsp} &\lesssim \eta r, \label{eq:BoundaryTermlinlin}\\
		\| \bdyop(\nlwave, \linschr)\|_{\schrsp} &\lesssim \eta \| \nlwave \|_{L^\infty L^2}, \label{eq:BoundaryTermnllin}\\
		\| \bdyop(\linwave, \nlschr)\|_{\schrsp} &\lesssim \eta \| \langle \nabla \rangle \nlschr \|_{L^\infty_\si L^2 \cap L^2_\si \dot{B}^0_{6,2}}, \label{eq:BoundaryTermlinnl}\\
		\| \bdyop(\nlwave, \nlschr)\|_{\schrsp} &\lesssim T^{-\si} \| \nlwave \|_{L^\infty_\si L^2} \| \langle \nabla \rangle \nlschr \|_{L^\infty_\si L^2 \cap L^2_\si \dot{B}^0_{6,2}} \label{eq:BoundaryTermnlnl}
	\end{align}
	for all $\nlschr \in L^\infty_\sigma(I_T, H^1(\R^3)) \cap \langle \nabla \rangle^{-1}L^2_\sigma(I_T, \dot{B}^0_{6,2}(\R^3))$ and $\nlwave \in L^\infty_\sigma(I_T, L^2(\R^3))$.
\end{lem}
	
	\begin{proof}
		We first note that $|\nabla| \langle \nabla \rangle \bdyop(\cdot \,, \cdot)$ is a bilinear multiplier whose symbol
		\begin{align*}
			m(\xi,\eta) = \frac{|\xi + \eta| \langle \xi + \eta \rangle \sum_{k \nsim \log_2 \alpha} \rho_k(\xi) \rho_{\leq k-5}(\eta)}{|\xi + \eta|^2 + \alpha |\xi| - |\eta|^2}
		\end{align*}
		satisfies the assumptions of Lemma~\ref{lem:BilinearMultiplierEstimate}, as a straightforward computation shows. 
		
		We start by proving~\eqref{eq:BoundaryTermlinlin}-\eqref{eq:BoundaryTermnlnl} for the $L^\infty_\si H^1$-component of the $\schrsp$-norm. For fixed $t \in [T, \infty)$, we employ dyadic decomposition, Sobolev's embedding and Lemma~\ref{lem:BilinearMultiplierEstimate} to estimate
		\begin{align*}
			&\| \langle \nabla \rangle \bdyop(v(t), \linschr(t)) \|_{L^2} 
			\lesssim \Big(\sum_{k \in \Z} \|\langle \nabla \rangle \bdyop(P_k v(t), P_{\leq k-5} \linschr(t))\|_{L^2}^2 \Big)^{\frac{1}{2}} \\
			&\lesssim \Big(\sum_{k \in \Z} \Big(\sum_{k_1 \leq k-5} \| |\nabla| \langle \nabla \rangle \bdyop(P_k v(t), P_{k_1} \linschr(t))\|_{L^{\frac{6}{5}}} \Big)^2 \Big)^{\frac{1}{2}} \\
			&\lesssim \Big(\sum_{k \in \Z} \Big(\sum_{k_1 \leq k-5} \|P_k v(t)\|_{L^{2}} \|P_{k_1} \linschr(t)\|_{L^3} \Big)^2 \Big)^{\frac{1}{2}} \\
			&\lesssim \|v(t)\|_{L^2} \sum_{k_1 \in \Z} (1+2^{2k_1})^{-\frac{1}{2}}  \|\langle \nabla \rangle P_{k_1} \linschr(t)\|_{L^3}.
		\end{align*}
		Taking the $L^\infty_\si$-norm and employing Lemma~\ref{lem:SobolevSchroedingerWave}~\ref{it:SobolevSchroedinger}, we thus obtain
		\begin{align*}
			&\|\bdyop(v, \linschr)\|_{L^\infty_\si H^1} \lesssim \|v\|_{L^\infty L^2} \sum_{k_1 \in \Z} (1+2^{2k_1})^{-\frac{1}{2}}  \|\langle \nabla \rangle P_{k_1} \linschr\|_{L^\infty_\si L^3} \\
			&\lesssim \|v\|_{L^\infty L^2} \sum_{k_1 \in \Z} (1+2^{2k_1})^{-\frac{1}{2}} 2^{k_1 \nu}  \|\langle \nabla \rangle P_{k_1} \linschr\|_{L^{\frac{2}{\nu}}_\si L^3}  \lesssim \eta \|v\|_{L^\infty L^2}.
		\end{align*}
		Setting $v = \linwave$ respectively $v = \nlwave$, we obtain the $L^\infty_\si H^1$-estimate in~\eqref{eq:BoundaryTermlinlin} and~\eqref{eq:BoundaryTermnllin}. Using dyadic decomposition and Lemma~\ref{lem:BilinearMultiplierEstimate} again, we get
		\begin{align*}
			&\| \langle \nabla \rangle \bdyop(v(t), \nlschr(t)) \|_{L^2} 
			\lesssim \Big(\sum_{k \in \Z} \|\langle \nabla \rangle \bdyop(P_k v(t), P_{\leq k-5} \nlschr(t))\|_{L^2}^2 \Big)^{\frac{1}{2}}  \\
			&\lesssim  \Big(\sum_{k \in \Z} \Big(\sum_{k_1 \leq k-5} \| |\nabla|^{-1} P_k v(t)\|_{L^6} \| P_{k_1} \nlschr(t))\|_{L^3}\Big)^2 \Big)^{\frac{1}{2}}
			\lesssim \| v(t) \|_{\dot{B}^{-1}_{6,2}} \|\nlschr(t)\|_{H^1}
		\end{align*}
		for fixed $t \in [T,\infty)$. Since $\| \nlwave(t) \|_{\dot{B}^{-1}_{6,2}} \lesssim \|\nlwave(t)\|_{L^2}$, we obtain
		\begin{align}
		\label{eq:EstimateOmeganlnlLinftyL2}
			&\| \bdyop(\nlwave, \nlschr) \|_{L^\infty_\si H^1} 
			\lesssim \|\nlwave\|_{L^\infty L^2} \| \nlschr \|_{L^\infty_\si H^1} \lesssim T^{-\si} \|\nlwave\|_{L^\infty_\si L^2} \|\nlschr\|_{L^\infty_\si H^1},
		\end{align}
		while $\| \linwave(t) \|_{\dot{B}^{-1}_{6,2}} \lesssim \| \linwave(t) \|_{\dot{B}^{-\frac{3}{4}}_{4,2}}$ combined with Lemma~\ref{lem:SobolevSchroedingerWave}~\ref{it:SobolevWave} yields
		\begin{align}
			\| \bdyop(\linwave, \nlschr) \|_{L^\infty_\si H^1} 
			&\lesssim \|\linwave\|_{L^\infty \dot{B}^{-\frac{3}{4}}_{4,2}} \|\nlschr\|_{L^\infty_\si H^1} \lesssim \|\linwave\|_{\dot{B}^{-\frac{1}{2}}_{4,4,2}} \|\nlschr\|_{L^\infty_\si H^1} \nonumber \\
			&\lesssim \eta  \|\nlschr\|_{L^\infty_\si H^1}. \label{eq:EstimateOmegalinnlLinftyL2}
		\end{align}
		We conclude that also~\eqref{eq:BoundaryTermlinnl} and~\eqref{eq:BoundaryTermnlnl} hold for the $L^\infty_\si H^1$-norm on the left-hand side.
		
		For the $\langle \nabla \rangle^{-1} L^2_\si \dot{B}^0_{6,2}$-component of the $\schrsp$-norm we estimate
		\begin{align*}
			&\| \langle \nabla \rangle \bdyop(v,u) \|_{L^2_\si \dot{B}^0_{6,2}} 
			\lesssim \Big\|\Big( \sum_{k \in \Z} \| \langle \nabla \rangle \bdyop(P_k v, P_{\leq k-5} u)\|_{L^6}^2 \Big)^{\frac{1}{2}}\Big\|_{L^2_\si} \\
			&\lesssim \Big\|\Big( \sum_{k \in \Z} \Big( \sum_{k_1 \leq k-5} \| |\nabla|^{-1} P_k v\|_{L^6} \| P_{k_1} u\|_{L^\infty}\Big)^2 \Big)^{\frac{1}{2}} \Big\|_{L^2_\si} \lesssim \| v \|_{L^\infty \dot{B}^{-1}_{6,2}} \|\langle \nabla \rangle u\|_{L^2_\si \dot{B}^0_{6,2}}.
		\end{align*}
		Arguing as in~\eqref{eq:EstimateOmeganlnlLinftyL2} and~\eqref{eq:EstimateOmegalinnlLinftyL2} for $v = \nlwave$ respectively $v = \linwave$ and replacing $u$ by $\linschr$ respectively $\nlschr$, we obtain~\eqref{eq:BoundaryTermlinlin}-\eqref{eq:BoundaryTermnlnl} for the $\langle \nabla \rangle^{-1} L^2_\si \dot{B}^0_{6,2}$-component of the norm on the left-hand side.
		
		It remains to treat the $\langle \nabla \rangle^{-1} L^2_\si \dot{B}^{\frac{1}{4}+\ep}_{(q(\ep),\frac{2}{1-\nu}),2}$-component of the $\schrsp$-norm. Here we first employ H{\"o}lder's inequality on the sphere and then Lemma~\ref{lem:BilinearMultiplierEstimate} to infer
		\begin{align*}
			&\| \langle \nabla \rangle \bdyop(v,u) \|_{L^2_\si \dot{B}^{\frac{1}{4}+\ep}_{(q(\ep),\frac{2}{1-\nu}),2}} 
			\lesssim \Big\|\Big( \sum_{k \in \Z} 2^{2k(\frac{1}{4}+\ep)}\| \langle \nabla \rangle \bdyop(P_k v, P_{\leq k-5} u)\|_{L^{q(\ep)}}^2 \Big)^{\frac{1}{2}}\Big\|_{L^2_\si} \\
			&\lesssim \Big\|\Big( \sum_{k \in \Z} \Big( \sum_{k_1 \leq k-5} \| |\nabla|^{-\frac{3}{4}+\ep} P_k v\|_{L^{q(\ep)}} \| P_{k_1} u\|_{L^{\infty}}\Big)^2 \Big)^{\frac{1}{2}}\Big\|_{L^2_\si} \\
			&\lesssim \| v \|_{L^\infty \dot{B}^{-\frac{3}{4}+\ep}_{q(\ep),2}} \|\langle \nabla \rangle u\|_{L^2_\si \dot{B}^0_{6,2}}.
		\end{align*}
		For $v = \linwave$, Lemma~\ref{lem:SobolevSchroedingerWave}~\ref{it:SobolevWave} implies
		\begin{align*}
			&\| \langle \nabla \rangle \bdyop(\linwave,u) \|_{L^2_\si \dot{B}^{\frac{1}{4}+\ep}_{(q(\ep),\frac{2}{1-\nu}),2}} 
			\lesssim \|\linwave\|_{\dot{B}^{-\frac{1}{2}+\frac{4}{3}\ep}_{q(\ep),q(\ep),2}}  \|\langle \nabla \rangle u\|_{L^2_\si \dot{B}^0_{6,2}}
			\lesssim \eta \|\langle \nabla \rangle u\|_{L^2_\si \dot{B}^0_{6,2}}.
		\end{align*}
		In the case $v = \nlwave$, we exploit $\|\nlwave\|_{L^\infty \dot{B}^{-\frac{3}{4}+\ep}_{q(\ep),2}} \lesssim \|\nlwave\|_{L^\infty L^2}$ to deduce
		\begin{align*}
			\| \langle \nabla \rangle \bdyop(\nlwave,u) \|_{L^2_\si \dot{B}^{\frac{1}{4}+\ep}_{(q(\ep),\frac{2}{1-\nu}),2}} 
			&\lesssim \| \nlwave \|_{L^\infty L^2} \|\langle \nabla \rangle u\|_{L^2_\si \dot{B}^0_{6,2}} \\
			&\lesssim T^{-\si} \| \nlwave \|_{L^\infty_\si L^2} \|\langle \nabla \rangle u\|_{L^2_\si \dot{B}^0_{6,2}}.
		\end{align*}
		Setting $u = \linschr$ respectively $u = \nlschr$, we conclude that~\eqref{eq:BoundaryTermlinlin}-\eqref{eq:BoundaryTermnlnl} also holds for the $\langle \nabla \rangle^{-1}  L^2_\si \dot{B}^{\frac{1}{4}+\ep}_{(q(\ep),\frac{2}{1-\nu}),2}$-component of the norm on the left-hand side.
	\end{proof}

	\subsection{Quadratic terms}
	We next estimate the quadratic nonlinearities. We recall that, as discussed in the introduction, the $\linwave \nlschr$ component of the quadratic nonlinearity is the most difficult term. To estimate it, we exploit the interplay of the physical-space and the angular randomization which shows itself in the available norms here.
	\begin{lem}
		\label{lem:BilinearEstimateSchroedinger}
		Fix $T \geq 1$ and let $r, \eta > 0$. Take $u_0, v_0 \in L^2(\R^3)$ with $\|u_0\|_{L^2} + \|v_0\|_{L^2} \leq r$ such that $\linschr(t) = \eul^{\imu t \Delta} u_0$ and $\linwave(t) = \eul^{\imu \alpha t |\nabla|} v_0$ satisfy
		\begin{align*}
			\| \langle \nabla \rangle \linschr \|_{L^{\frac{2}{2-\nu}}_\sigma \dot{B}^0_{\frac{3}{\nu},2}(I_T)} +
			\| \linwave \|_{L^2 \dot{B}^{-\frac{1}{4}-\ep}_{(q(-\ep),\frac{2}{\nu}),2}(I_T)} \leq \eta.
		\end{align*}
		We then have the estimates
		\begin{align}
			\Big\|\int_{t}^\infty \eul^{\imu (t-s) \Delta} (\linwave \linschr)_{R} \dd s \Big\|_{\schrsp} &\lesssim \eta r,\label{eq:EstimateBilinearSchroedingerlinlin}\\
			\Big\|\int_{t}^\infty \eul^{\imu (t-s) \Delta} (\nlwave \linschr)_{R} \dd s \Big\|_{\schrsp} &\lesssim \eta \| \nlwave \|_{L^\infty L^2}, \label{eq:EstimateBilinearSchroedingernllin}\\
			\Big\|\int_{t}^\infty \eul^{\imu (t-s) \Delta} (\linwave \nlschr)_{R} \dd s \Big\|_{\schrsp} &\lesssim \eta \| \langle \nabla \rangle \nlschr \|_{L^2_\si \dot{B}^{\frac{1}{4}+\ep}_{(q(\ep),\frac{2}{1-\nu}),2}}, \label{eq:EstimateBilinearSchroedingerlinnl}\\
			\Big\|\int_{t}^\infty \eul^{\imu (t-s) \Delta} (\nlwave \nlschr)_{R} \dd s \Big\|_{\schrsp} &\lesssim T^{-\nu} \| \nlwave \|_{L^\infty_\si L^2} \| \langle \nabla \rangle \nlschr \|_{L^2_\si \dot{B}^0_{6,2}} \label{eq:EstimateBilinearSchroedingernlnl}
		\end{align}
		for all $\nlschr \in \langle \nabla \rangle^{-1}L^2_\sigma(I_T, \dot{B}^0_{6,2}(\R^3)) \cap \langle \nabla \rangle^{-1} L^2_\sigma(I_T,\dot{B}^{\frac{1}{4} + \ep}_{(q(\ep),\frac{2}{1-\nu}),2}(\R^3))$ and $\nlwave \in L^\infty_\sigma(I_T, L^2(\R^3))$.
	\end{lem}
	
	\begin{proof}
	 We first prove estimates~\eqref{eq:EstimateBilinearSchroedingerlinlin}-\eqref{eq:EstimateBilinearSchroedingernlnl} for the $LH$-component of the resonant part. We start with~\eqref{eq:EstimateBilinearSchroedingerlinlin}.
	 
	 Proposition~\ref{prop:TimeWeightedStrichartzEstimatesSchroedinger} allows us to estimate
	 \begin{align*}
	 	&\Big\|\int_{t}^\infty \eul^{\imu (t-s) \Delta} (\linwave \linschr)_{LH} \dd s \Big\|_{\schrsp} \lesssim \| \langle \nabla \rangle (\linwave \linschr)_{LH} \|_{L^{\frac{2}{2-\nu}}_\si \dot{B}^0_{\frac{6}{3+2\nu},2}} \\
	 	&\lesssim \Big\|\Big(\sum_{k \in \Z}(1+2^{2k}) \| P_{\leq k-5} \linwave P_k \linschr \|_{L^{\frac{6}{3 + 2\nu}}}^2 \Big)^{\frac{1}{2}} \Big\|_{L^{\frac{2}{2-\nu}}_\si} \lesssim \|\linwave\|_{L^\infty L^2} \|\langle \nabla \rangle \linschr \|_{L^{\frac{2}{2-\nu}}_\si \dot{B}^0_{\frac{3}{\nu},2}} \\ &\lesssim \eta r.
	 \end{align*}
	 Replacing $\linwave$ by $\nlwave$ in the above estimate, we also obtain~\eqref{eq:EstimateBilinearSchroedingernllin} for the $LH$-component.
	 
	 We proceed with~\eqref{eq:EstimateBilinearSchroedingerlinnl}. Using Proposition~\ref{prop:TimeWeightedStrichartzEstimatesSchroedinger}, we find for the $LH$-component
	 \begin{align*}
	 	&\Big\|\int_{t}^\infty \eul^{\imu (t-s) \Delta} (\linwave \nlschr)_{LH} \dd s \Big\|_{\schrsp}
	 	 \lesssim \| \langle \nabla \rangle (\linwave \nlschr)_{LH} \|_{L^1_\si L^2} \\
	 	 &\lesssim \Big\|\Big(\sum_{k \in \Z} (1+2^{2k}) \|P_{\leq k-5} \linwave P_k \nlschr \|_{L^2}^2 \Big)^{\frac{1}{2}} \Big\|_{L^1_\si}  \\
	 	&\lesssim \Big\|\Big(\sum_{k \in \Z} (1+2^{2k})\| P_k \nlschr \|_{\cL^{q(\ep)}_r L^{\frac{2}{1-\nu}}_\theta}^2 \\
	 	&\hspace{10em} \cdot \Big(\sum_{k_1 \leq k-5} 2^{k_1(-\frac{1}{4}-\ep)} \|P_{k_1} \linwave \|_{\cL^{q(-\ep)}_r L^{\frac{2}{\nu}}_\theta} 2^{k_1(\frac{1}{4}+\ep)} \Big)^2  \Big)^{\frac{1}{2}} \Big\|_{L^1_\si} \\
	 	&\lesssim \Big\| \| \linwave \|_{\dot{B}^{-\frac{1}{4}-\ep}_{(q(-\ep),\frac{2}{\nu}),2}}\Big(\sum_{k \in \Z} 2^{2k(\frac{1}{4}+\ep)}(1+2^{2k})\| P_k \nlschr \|_{\cL^{q(\ep)}_r L^{\frac{2}{1-\nu}}_\theta}^2   \Big)^{\frac{1}{2}} \Big\|_{L^1_\si} \\
	 	&\lesssim \| \linwave \|_{L^2 \dot{B}^{-\frac{1}{4}-\ep}_{(q(-\ep),\frac{2}{\nu}),2}} \| \langle \nabla \rangle \nlschr \|_{L^2_\si \dot{B}^{\frac{1}{4}+\ep}_{(q(\ep),\frac{2}{1-\nu}),2}}\lesssim \eta \| \langle \nabla \rangle \nlschr \|_{L^2_\si \dot{B}^{\frac{1}{4}+\ep}_{(q(\ep),\frac{2}{1-\nu}),2}}.
	 \end{align*}
To prove~\eqref{eq:EstimateBilinearSchroedingernlnl} for the $LH$-component, we again apply Proposition~\ref{prop:TimeWeightedStrichartzEstimatesSchroedinger} to derive
	\begin{align*}
		&\Big\|\int_{t}^\infty \eul^{\imu (t-s) \Delta} (\nlwave \nlschr)_{LH} \dd s \Big\|_{\schrsp}
	 	 \lesssim \| \langle \nabla \rangle (\nlwave \nlschr)_{LH} \|_{L^{\frac{4}{3}}_\si \dot{B}^0_{\frac{3}{2},2}} \\
	 	  &\lesssim \Big\|\Big(\sum_{k \in \Z} (1+2^{2k}) \|P_{\leq k-5} \nlwave\|_{L^2}^2 \|P_k \nlschr \|_{L^6}^2 \Big)^{\frac{1}{2}} \Big\|_{L^{\frac{4}{3}}_\si}
	 	  \lesssim \| \nlwave \|_{L^4 L^2} \| \langle \nabla \rangle \nlschr \|_{L^2_\si \dot{B}^0_{6,2}}\\
	 	 &\lesssim T^{-\nu} \| \nlwave \|_{L^\infty_\si L^2} \|\langle \nabla \rangle \nlschr \|_{L^2_\si \dot{B}^0_{6,2}}.
	\end{align*}	 
	
	Straightforward adaptions of the above estimates also yield~\eqref{eq:EstimateBilinearSchroedingerlinlin}-\eqref{eq:EstimateBilinearSchroedingernlnl} for the $HH$-component as well as~\eqref{eq:EstimateBilinearSchroedingerlinlin}, \eqref{eq:EstimateBilinearSchroedingernllin}, and~\eqref{eq:EstimateBilinearSchroedingernlnl} for the $\alpha L$-component. E.g., we estimate
	\begin{align*}
		&\Big\|\int_{t}^\infty \eul^{\imu (t-s) \Delta} (\linwave \linschr)_{HH} \dd s \Big\|_{\schrsp} \lesssim \| \langle \nabla \rangle (\linwave \linschr)_{HH} \|_{L^{\frac{2}{2-\nu}}_\si \dot{B}^0_{\frac{6}{3+2\nu},2}}  \\
	 	&\lesssim \| \langle \nabla \rangle (\linwave \linschr)_{HH} \|_{L^{\frac{2}{2-\nu}}_\si L^{\frac{6}{3+2\nu}}} \lesssim \Big\|\sum_{k \in \Z}\sum_{l = -4}^4 (1+2^{k}) \| P_{k+l} \linwave \|_{L^2} \| P_k \linschr \|_{L^{\frac{3}{\nu}}}\Big\|_{L^{\frac{2}{2-\nu}}_\si} \\
	 	&\lesssim \|\linwave\|_{L^\infty L^2} \|\langle \nabla \rangle \linschr \|_{L^{\frac{2}{2-\nu}}_\si \dot{B}^0_{\frac{3}{\nu},2}} \lesssim \eta r,\\
	 	&\Big\|\int_{t}^\infty \eul^{\imu (t-s) \Delta} (\linwave \linschr)_{\alpha L} \dd s \Big\|_{\schrsp} \lesssim \| \langle \nabla \rangle (\linwave \linschr)_{\alpha L} \|_{L^{\frac{2}{2-\nu}}_\si \dot{B}^0_{\frac{6}{3+2\nu},2}} \\
	 	&\lesssim \Big\| \Big(\sum_{k \sim \log_2 \alpha}(1+2^{2k}) \| P_{k} \linwave \|_{L^2}^2 \| P_{\leq k-5} \linschr \|_{L^{\frac{3}{\nu}}}^2 \Big)^{\frac{1}{2}} \Big\|_{L^{\frac{2}{2-\nu}}_\si} \\
	 	&\lesssim \|\linwave\|_{L^\infty L^2} \|\linschr \|_{L^{\frac{2}{2-\nu}}_\si L^{\frac{3}{\nu}}} \lesssim \eta r
	\end{align*}	
	for~\eqref{eq:EstimateBilinearSchroedingerlinlin}.
	 Only for the $\alpha L$-component in~\eqref{eq:EstimateBilinearSchroedingerlinnl} we have to use different exponents.
	We set $\frac{1}{\tilde{q}} = \frac{1}{2} - \frac{3}{4} \nu$ and $\frac{1}{\tilde{r}} = \frac{1}{4} + \frac{\ep}{3} + \frac{\nu}{2}$. Proposition~\ref{prop:TimeWeightedStrichartzEstimatesSchroedinger}~\ref{it:TimeWeightedStrichartzextext} then yields
	\begin{align*}
		&\Big\|\int_{t}^\infty \eul^{\imu (t-s) \Delta} (\linwave \nlschr)_{\alpha L} \dd s \Big\|_{\schrsp}
	 	 \lesssim \| \langle \nabla \rangle (\linwave \nlschr)_{\alpha L}\|_{L^{\tilde{q}'}_\si \dot{B}^{\frac{3}{2}-\frac{2}{\tilde{q}}-\frac{3}{\tilde{r}}}_{(\tilde{r}',2),2}} \\
	 	 &\lesssim \Big\|\Big(\sum_{k \sim \log_2 \alpha} (1+2^{2k})2^{2k(\frac{3}{2}-\frac{2}{\tilde{q}}-\frac{3}{\tilde{r}})} \| P_k \linwave P_{\leq k-5} \nlschr \|_{\mathcal{L}^{\tilde{r}'}_r L^2_\theta}^2 \Big)^{\frac{1}{2}} \Big\|_{L^{\tilde{q}'}_\si} \\
	 	 &\lesssim \Big\|\Big(\sum_{k \sim \log_2 \alpha} \| P_k \linwave\|_{\mathcal{L}^{q(-\ep)}_{r}L^{\frac{2}{\nu}}_\theta}^2 \|P_{\leq k-5} \nlschr \|_{\mathcal{L}^{\frac{2}{1-\nu}}_r L^{\frac{2}{1-\nu}}_\theta}^2 \Big)^{\frac{1}{2}} \Big\|_{L^{\tilde{q}'}_\sigma}   \\
	 	 &\lesssim \|\linwave\|_{L^2 \dot{B}^{-\frac{1}{4}-\ep}_{(q(-\ep),\frac{2}{\nu}),2}} \|\nlschr \|_{L_\si^{\frac{4}{3\nu}}L^{\frac{2}{1-\nu}}} \lesssim \eta \| \nlschr \|_{L^\infty_\si L^2 \cap L^2_\si L^6},
	\end{align*}
	where we used interpolation in the last step.
	\end{proof}
	
	The nonlinearity of the wave component is easier to control as the time weight we gain through the physical-space randomization is sufficient to close the estimate.
	\begin{lem}
		\label{lem:BilinearEstimateWave}
		Fix $T \geq 1$ and let $\eta > 0$. Take $u_0 \in L^2(\R^3)$ such that $\linschr(t) = \eul^{\imu t \Delta} u_0$ satisfies
		\begin{align*}
			\| \langle \nabla \rangle \linschr \|_{L^{\frac{2}{2-\nu}}_\sigma \dot{B}^0_{\frac{3}{\nu},2}(I_T)} \leq \eta.
		\end{align*}
		We then have the estimates
		\begin{align}
			\label{eq:BilinearEstimateWave}
			&\Big\|\int_{t}^\infty \eul^{\imu \alpha (t-s) |\nabla|} |\nabla| (u \overline{\linschr} + \linschr \overline{u}) \dd s \Big\|_{L^\infty_{\sigma}L^2} \lesssim \eta  \| \langle \nabla \rangle u \|_{L^\infty L^2 \cap L^2 \dot{B}^0_{6,2}}, \\
			&\Big\|\int_{t}^\infty \eul^{\imu \alpha (t-s) |\nabla|} |\nabla| \big(\nlschr^1 \overline{\nlschr^2}\big) \dd s \Big\|_{L^\infty_{\sigma}L^2} \lesssim T^{-\nu} \| \langle \nabla \rangle \nlschr^1 \|_{L_\sigma^2 \dot{B}^0_{6,2}} \| \langle \nabla \rangle \nlschr^2 \|_{L^\infty_\si L^2 \cap L^2_\si \dot{B}^0_{6,2}}
		\end{align}
		for all $u \in L^\infty(I_T, H^1(\R^3)) \cap \langle \nabla \rangle^{-1} L^2(I_T, \dot{B}^0_{6,2}(\R^3))$, $\nlschr^1 \in \langle \nabla \rangle^{-1} L^2_\sigma(I_T, \dot{B}^0_{6,2}(\R^3))$ and $\nlschr^2 \in L^\infty_\sigma(I_T, H^1(\R^3)) \cap \langle \nabla \rangle^{-1} L^2_\sigma(I_T, \dot{B}^0_{6,2}(\R^3))$.
	\end{lem}
	
	\begin{proof}
		Recalling~\eqref{eq:TimeWeightedEnergyEstimateWave} and using dyadic decomposition, we derive
		\begin{align*}
			&\Big\|\int_{t}^\infty \eul^{\imu \alpha (t-s) |\nabla|} |\nabla| (u \overline{\linschr} + \linschr \overline{u}) \dd s \Big\|_{L^\infty_{\sigma}L^2} \lesssim \| |\nabla| (u \overline{\linschr} + \linschr \overline{u})\|_{L^1_\si L^2} \\
			&\lesssim \| \| \langle \nabla \rangle u \|_{\dot{B}^0_{\frac{6}{3-2\nu},2}} \|\langle \nabla \rangle \linschr \|_{\dot{B}^0_{\frac{3}{\nu},2}}\|_{L^1_\si} 
			\lesssim \| \langle \nabla \rangle \linschr \|_{L^{\frac{2}{2-\nu}}_\sigma \dot{B}^0_{\frac{3}{\nu},2}} \| \langle \nabla \rangle u \|_{L^{\frac{2}{\nu}} \dot{B}^0_{\frac{6}{3-2\nu},2}} \\
			&\lesssim \eta \| \langle \nabla \rangle u \|_{L^\infty L^2 \cap L^2 \dot{B}^0_{6,2}},
		\end{align*}
		where we interpolated in the last step. For the second estimate, we again employ~\eqref{eq:TimeWeightedEnergyEstimateWave}, dyadic decomposition, and interpolation to infer
		\begin{align*}
			&\Big\|\int_{t}^\infty \eul^{\imu \alpha (t-s) |\nabla|} |\nabla| \big(\nlschr^1 \overline{\nlschr^2}\big) \dd s \Big\|_{L^\infty_{\sigma}L^2} \lesssim \big\| |\nabla| \big(\nlschr^1 \overline{\nlschr^2}\big) \big\|_{L^1_\si L^2} \\
			&\lesssim \| \langle \nabla \rangle \nlschr^1 \|_{L^2_\si \dot{B}^0_{6,2}} \| \langle \nabla \rangle \nlschr^2 \|_{L^2 \dot{B}^0_{3,2} } 
			 \lesssim T^{-\nu}\| \langle \nabla \rangle \nlschr^1 \|_{L^2_\si \dot{B}^0_{6,2}} \| \langle \nabla \rangle \nlschr^2 \|_{L^4_\si \dot{B}^0_{3,2} } \\
			 &\lesssim T^{-\nu} \| \langle \nabla \rangle \nlschr^1 \|_{L^2_\si \dot{B}^0_{6,2}} \| \langle \nabla \rangle \nlschr^2 \|_{L^\infty_\si L^2 \cap L^2_\si \dot{B}^0_{6,2} }. \qedhere
		\end{align*}
	\end{proof}
	
	\subsection{Cubic terms}
	
	It remains to control the cubic nonlinearities. Using time-weighted spaces, the expectation is that higher nonlinearities should be easier to control as we gain more time decay. This expectation is justified as the following lemmas show.
	\begin{lem}
		\label{lem:EstimateCubicTermsuuu}
		Fix $T \geq 1$. Then the estimate
		\begin{align*}
			&\Big\|\int_{t}^\infty  \eul^{\imu (t-s) \Delta} \bdyop(|\nabla| \big(u_1 \overline{u_2}), u_3 \big) \dd s \Big\|_{\schrsp} \\
			&\hspace{12em} \lesssim T^{-\si} \| \langle \nabla \rangle u_1 \|_{L^2_\si \dot{B}^0_{6,2}} \| u_2 \|_{L^\infty L^2} \| \langle \nabla \rangle u_3 \|_{L^2_\si \dot{B}^0_{6,2}}
		\end{align*}
		holds for all $u_1, u_3 \in \langle \nabla \rangle ^{-1} L^2_\si(I_T, \dot{B}^0_{6,2}(\R^3))$ and $ u_2 \in L^\infty(I_T, L^2(\R^3))$.
	\end{lem}
	\begin{proof}
		We first note that Proposition~\ref{prop:TimeWeightedStrichartzEstimatesSchroedinger} gives
		\begin{align*}
			\Big\|\int_{t}^\infty  \eul^{\imu (t-s) \Delta} \bdyop(|\nabla| \big(u_1 \overline{u_2}), u_3 \big) \dd s \Big\|_{\schrsp} \lesssim \|\bdyop(|\nabla| \big(u_1 \overline{u_2}), u_3 )\|_{L^1_\si H^1}.
		\end{align*}
		Combining dyadic decomposition with Lemma~\ref{lem:BilinearMultiplierEstimate} as in the proof of Lemma~\ref{lem:BilinearEstimateSchroedinger}, we derive
		\begin{align*}
			&\|\bdyop(|\nabla| (u_1 \overline{u_2}), u_3)\|_{L^1_\si H^1}
			\lesssim \Big\| \Big(\sum_{k \in \Z} \| \langle \nabla \rangle \bdyop(P_k |\nabla| (u_1 \overline{u_2}), P_{\leq k-5} u_3)\|_{L^2}^2 \Big)^{\frac{1}{2}}\Big\|_{L^1_\si} \\
			&\lesssim \Big\| \Big(\sum_{k \in \Z} \Big( \sum_{k_1 \leq k-5} \|P_k (u_1 \overline{u_2})\|_{L^2} \|P_{k_1} u_3\|_{L^\infty} \Big)^2 \Big)^{\frac{1}{2}} \Big\|_{L^1_\si}
			\lesssim \| \|u_1 \overline{u_2}\|_{L^2} \|\langle \nabla \rangle u_3\|_{\dot{B}^0_{6,2}} \|_{L^1_\si} \\
			&\lesssim T^{-\si} \| \langle \nabla \rangle u_1\|_{L^2_\si \dot{B}^0_{6,2}} \|u_2\|_{L^\infty L^2} \|\langle \nabla \rangle u_3\|_{L^2_\si \dot{B}^0_{6,2}}. \qedhere
		\end{align*}
	\end{proof}
	
	The second cubic nonlinearity needs a bit more effort as we have to distinguish between $\linwave$ and $\nlwave$.
	\begin{lem}
		\label{lem:EstimateCubicTermsVVu}
		Fix $T \geq 1$ and let $0 < \eta \leq r$. Take $u_0, v_0 \in L^2(\R^3)$ with $\|u_0\|_{L^2} + \|v_0\|_{L^2} \leq r$ such that $\linschr(t) = \eul^{\imu t \Delta} u_0$ and $\linwave(t) = \eul^{\imu \alpha t |\nabla|} v_0$ satisfy
		\begin{align*}
			\| \langle \nabla \rangle \linschr \|_{L^{2}_\sigma \dot{B}^0_{6,2}(I_T)} +
			\| \linwave \|_{L^4 \dot{B}^{-\frac{1}{2}}_{4,2}(I_T)} \leq \eta.
		\end{align*}
		We then have the estimates
		\begin{align}
			&\Big\|\int_{t}^\infty \eul^{\imu (t-s) \Delta} \bdyop(\linwave,\linwave u) \dd s \Big\|_{\schrsp} \lesssim \eta r \| \langle \nabla \rangle u\|_{L^2_\si \dot{B}^0_{6,2}}, \label{eq:Cubiclinlin}\\
			&\Big\|\int_{t}^\infty \eul^{\imu (t-s) \Delta} \bdyop(v,\nlwave u) \dd s \Big\|_{\schrsp} + \Big\|\int_{t}^\infty \eul^{\imu (t-s) \Delta} \bdyop(\nlwave,v u) \dd s \Big\|_{\schrsp} \nonumber\\
			&\hspace{14em}\lesssim T^{-\nu} \|v\|_{L^\infty L^2} \|\nlwave\|_{L^\infty_\si L^2} \| \langle \nabla \rangle u\|_{L^2_\si \dot{B}^0_{6,2}} \label{eq:Cubicnlwave}
		\end{align}
		for all $u \in \langle \nabla \rangle^{-1} L^2_\si(I_T, \dot{B}^0_{6,2}(\R^3))$, $\nlwave \in L^\infty_\sigma(I_T,L^2(\R^3))$, and $v \in L^\infty(I_T, L^2(\R^3))$.
	\end{lem}
	
	\begin{proof}
		Let $v_1, v_2 \in L^\infty L^2$. Applying Proposition~\ref{prop:TimeWeightedStrichartzEstimatesSchroedinger}, we find
		\begin{align*}
			\Big\|\int_{t}^\infty \eul^{\imu (t-s) \Delta} \bdyop(v_1,v_2 u) \dd s \Big\|_{\schrsp}
			\lesssim \| \langle \nabla \rangle \bdyop(v_1,v_2 u)\|_{L^{\frac{4}{3}}_\si \dot{B}^0_{\frac{3}{2},2}}.
		\end{align*}
		Assume that $v_1 = \nlwave$, $v_2 = v$ or $v_1 = v$, $v_2 = \nlwave$.
	Dyadic decomposition and Lemma~\ref{lem:BilinearMultiplierEstimate} then yield
		\begin{align*}
			&\| \langle \nabla \rangle \bdyop(v_1,v_2 u)\|_{L^{\frac{4}{3}}_\si \dot{B}^0_{\frac{3}{2},2}}
			\lesssim \Big\| \Big(\sum_{k \in \Z} \Big(\sum_{k_1 \leq k-5} 2^{-k} \|P_k v_1 \|_{L^2} \|P_{k_1} (v_2 u)\|_{L^6} \Big)^2 \Big)^{\frac{1}{2}} \Big\|_{L^{\frac{4}{3}}_\si} \\
			&\lesssim  \Big\| \Big(\sum_{k \in \Z} \|P_k v_1 \|_{L^2}^2 \Big(\sum_{k_1 \leq k-5} 2^{-k + k_1}  \|P_{k_1} (v_2 u)\|_{L^2} \Big)^2 \Big)^{\frac{1}{2}} \Big\|_{L^{\frac{4}{3}}_\si} \\
			&\lesssim \Big\| \|v_1\|_{L^2} \|v_2\|_{L^2} \|\langle \nabla \rangle u\|_{L^6} \Big\|_{L^{\frac{4}{3}}_\si}
			\lesssim \|v\|_{L^\infty L^2} \|\nlwave\|_{L^4 L^2} \|\langle \nabla \rangle u\|_{L^2_\si \dot{B}^0_{6,2}} \\
			&\lesssim T^{-\nu}  \|v\|_{L^\infty L^2} \|\nlwave\|_{L^\infty_\si L^2} \|\langle \nabla \rangle u\|_{L^2_\si \dot{B}^0_{6,2}},
		\end{align*}
		which shows~\eqref{eq:Cubicnlwave}. In the case $v_1 = v_2 = \linwave$ we again employ dyadic decomposition and Lemma~\ref{lem:BilinearMultiplierEstimate} to infer
		\begin{align*}
			&\| \langle \nabla \rangle \bdyop(\linwave,\linwave u)\|_{L^{\frac{4}{3}}_\si \dot{B}^0_{\frac{3}{2},2}}
			\lesssim \Big\| \Big(\sum_{k \in \Z} \Big(\sum_{k_1 \leq k-5} 2^{-k} \|P_k \linwave \|_{L^4} \|P_{k_1} (\linwave u)\|_{L^{\frac{12}{5}}} \Big)^2 \Big)^{\frac{1}{2}} \Big\|_{L^{\frac{4}{3}}_\si} \\
			&\lesssim \Big\| \Big(\sum_{k \in \Z} 2^{-k} \|P_k \linwave\|_{L^4}^2 \Big(\sum_{k_1 \leq k-5} 2^{-\frac{k}{2} + \frac{k_1}{2}}  \|P_{k_1} (\linwave u)\|_{L^{\frac{12}{7}}} \Big)^2 \Big)^{\frac{1}{2}} \Big\|_{L^{\frac{4}{3}}_\si} \\
			&\lesssim \Big\| \|\linwave\|_{\dot{B}^{-\frac{1}{2}}_{4,2}} \|\linwave\|_{L^2} \| u \|_{L^{12}} \Big\|_{L^{\frac{4}{3}}_\si} \lesssim \|\linwave\|_{L^4 \dot{B}^{-\frac{1}{2}}_{4,2}} \|\linwave \|_{L^\infty L^2} \|\langle \nabla \rangle u \|_{L^2_\si \dot{B}^0_{6,2}} \\
			&\lesssim \eta r  \|\langle \nabla \rangle u \|_{L^2_\si \dot{B}^0_{6,2}},
		\end{align*}
		implying~\eqref{eq:Cubiclinlin}.
	\end{proof}

 \section{Proof of the main results}
 \label{sec:MainResults}
 
 We are now able to prove the main results of this article. We start by a deterministic version for the unique solvability of the final-state problem. It follows from the estimates from Section~\ref{sec:MultilinearEstimates} and a standard fixed point argument.
 \begin{prop}
 	\label{prop:FinalStateDeterministic}
 	Let $0 < \nu \ll 1$, $\sigma = \frac{1}{2} - \nu$, and $r > 0$.
 	
 	Then there exists $\eta = \eta(r) > 0$ such that  for all $u_+ \in H^1(\R^3)$ and $v_+ \in L^2(\R^3)$ such that $\|u_+\|_{H^1} + \|v_+\|_{L^2} \leq r$ and $\linschr(t) = \eul^{\imu t \Delta} u_+$ and $\linwave(t) = \eul^{\imu \alpha t |\nabla|} v_+$ satisfy
 	\begin{align*}
 		\|\langle \nabla \rangle \linschr\|_{L^{\frac{2}{\nu}}_\sigma \dot{B}^0_{3,2} \cap L^2_\sigma \dot{B}^0_{6,2} \cap L^{\frac{2}{2-\nu}}_\sigma \dot{B}^0_{\frac{3}{\nu},2}(I_{T_0})} \, + \, 
 		\| \linwave \|_{L^2 \dot{B}^{-\frac{1}{4}-\ep}_{(q(-\ep),\frac{2}{\nu}),2} \cap \dot{B}^{-\frac{1}{2}+\frac{4}{3}\ep}_{q(\ep), q(\ep),2} \cap \dot{B}^{-\frac{1}{2}}_{4,4,2}(I_{T_0})} \leq \eta
 	\end{align*}
 	for a time $T_0 \geq 1$, there is a time $T = T(r) \geq T_0$ and a unique solution $(u,v) \in C(I_T, H^1(\R^3)) \cap C(I_T, L^2(\R^3))$ of~\eqref{eq:ZakharovSystemFirstOrder} with
 	\begin{align*}
 		\| u - \linschr \|_{\schrsp} + \| v - \linwave \|_{\wavesp} \leq r.
 	\end{align*}
 \end{prop}
 
 \begin{proof}
 	We set $\nlschr = u - \linschr$ and $\nlwave = v - \linwave$. The reformulation~\eqref{eq:ZakharovSystemNormalForm} of the Zakharov system then becomes 
 	\begin{equation}
		\label{eq:ZakharovSystemNormalFormnl}
		\begin{aligned}
			\nlschr(t) &=  - \bdyop(\linwave + \nlwave, \linschr + \nlschr)(t) \\
						&\qquad + \imu \int_t^\infty \eul^{\imu (t-s) \Delta} ((\linwave + \nlwave)(\linschr + \nlschr))_{LH+HH+\alpha L}(s) \dd s \\
			&\qquad - \imu \alpha \int_t^\infty \eul^{\imu (t-s) \Delta} \bdyop(|\nabla| |\linschr + \nlschr|^2,\linschr + \nlschr)(s) \dd s \\
			&\qquad + \imu \int_t^\infty \eul^{\imu (t-s) \Delta} \bdyop(\linwave + \nlwave, (\linwave + \nlwave)(\linschr + \nlschr))(s) \dd s, \\
			\nlwave(t) &=  - \imu \alpha \int_t^\infty \eul^{\imu \alpha (t-s) |\nabla|}(|\nabla| |\linschr + \nlschr|^2)(s) \dd s,
		\end{aligned}
	\end{equation}	
	i.e. a fixed point equation for $(\nlschr, \nlwave)$. The estimates from Section~\ref{sec:MultilinearEstimates}, i.e., Lemmas~\ref{lem:EstimateBoundaryTerms} to~\ref{lem:EstimateCubicTermsVVu}, imply that the fixed point operator $\Phi$, defined by the right-hand side of~\eqref{eq:ZakharovSystemNormalFormnl}, is a contractive self-mapping on 
	\begin{align*}
		B_{r} = \{(\nlschr, \nlwave) \in \schrsp(I_T \times \R^3) \times \wavesp(I_T \times \R^3) \colon \|\nlschr\|_{\schrsp} + \|\nlwave\|_{\wavesp} \leq r\}
	\end{align*}
	for $\eta = \eta(r)$ sufficiently small and $T = T(r) \geq 1$ sufficiently large. The assertion now follows from Banach's fixed point theorem.
 \end{proof}
 
 The main theorem now follows from the randomization improved estimates from Section~\ref{sec:LinearRandomizedEstimates}. They imply that the linear solutions with randomized final states satisfy the assumptions of the previous proposition.
 \vspace{1em}
 
 \emph{Proof of Theorem~\ref{thm:MainResult}.} Recall that for given final states $u_+ \in H^1(\R^3)$ and $v_+ \in L^2(\R^3)$ we denote by $u_+^\omega$ the physical-space randomization of $u_+$ and by $v_+^\omega$ the angular randomization of $v_+$ on a probability space $(\Omega,\cA,\PP)$. We set $\sigma = \frac{1}{2} - \nu$.
 
 By Lemma~\ref{lem:LargeDeviationEstimate} and Corollary~\ref{cor:EstimatePkpsill2kl2lLp}, Proposition~\ref{prop:RandomizationImprovedEstimatesSchroedinger}, and Proposition~\ref{prop:RandomizationImprovedEstimatesHalfWave}, we have
 \begin{align*}
 	&\| u_+^\omega \|_{L^\beta_\omega H^1} + \|\langle \nabla \rangle \eul^{\imu t \Delta} u_+^\omega \|_{L^\beta_\omega L^{\frac{2}{\nu}}_\sigma \dot{B}^0_{3,2} \cap L^\beta_\omega L^2_\sigma \dot{B}^0_{6,2} \cap L^\beta_\omega L^{\frac{2}{2-\nu}}_\sigma \dot{B}^0_{\frac{3}{\nu},2}(I_1)} \lesssim \sqrt{\beta} \|u_+\|_{H^1}, \\
 	&\| v_+^\omega \|_{L^\beta_\omega L^2} + \| \eul^{\imu \alpha t |\nabla|} v_+^\omega \|_{L^\beta_\omega L^2 \dot{B}^{-\frac{1}{4}-\ep}_{(q(-\ep),\frac{2}{\nu}),2} \cap L^\beta_\omega \dot{B}^{-\frac{1}{2}+\frac{4}{3}\ep}_{q(\ep), q(\ep),2} \cap L^\beta_\omega \dot{B}^{-\frac{1}{2}}_{4,4,2}(I_1)} \lesssim \sqrt{\beta} \|v_+\|_{L^2}
 \end{align*}
 for all $\beta \geq 2$. In particular, there is a measurable set $\tilde{\Omega} \subseteq \Omega$ with $\PP(\tilde{\Omega}) = 1$ and
 \begin{align*}
 	&\| u_+^\omega \|_{H^1} + \|\langle \nabla \rangle \eul^{\imu t \Delta} u_+^\omega \|_{L^{\frac{2}{\nu}}_\sigma \dot{B}^0_{3,2} \cap L^2_\sigma \dot{B}^0_{6,2} \cap L^{\frac{2}{2-\nu}}_\sigma \dot{B}^0_{\frac{3}{\nu},2}(I_1)} < \infty, \\
 	&\| v_+^\omega \|_{L^2} + \| \eul^{\imu \alpha t |\nabla|} v_+^\omega \|_{L^2 \dot{B}^{-\frac{1}{4}-\ep}_{(q(-\ep),\frac{2}{\nu}),2} \cap \dot{B}^{-\frac{1}{2}+\frac{4}{3}\ep}_{q(\ep), q(\ep),2} \cap \dot{B}^{-\frac{1}{2}}_{4,4,2}(I_1)} < \infty
 \end{align*}
 for all $\omega \in \tilde{\Omega}$. Fix $\omega \in \tilde{\Omega}$ in the following. Take $r > 0$ such that
 \begin{align*}
 	\|u_+^\omega\|_{H^1} + \|v_+^\omega\|_{L^2} \leq r
 \end{align*}
 and choose $\eta = \eta(r)$ from Proposition~\ref{prop:FinalStateDeterministic}. Using dominated convergence, we find a time $T \geq 1$ so that
 \begin{align}
 \label{eq:CheckingAssumptionForDeterministicFinalState}
 	&\|\langle \nabla \rangle \eul^{\imu t \Delta} u_+^\omega \|_{L^{\frac{2}{\nu}}_\sigma \dot{B}^0_{3,2} \cap L^2_\sigma \dot{B}^0_{6,2} \cap L^{\frac{2}{2-\nu}}_\sigma \dot{B}^0_{\frac{3}{\nu},2}(I_T)}  \nonumber\\
 	&\qquad +  
 	\| \eul^{\imu \alpha t |\nabla|} v_+^\omega \|_{L^2 \dot{B}^{-\frac{1}{4}-\ep}_{(q(-\ep),\frac{2}{\nu}),2} \cap \dot{B}^{-\frac{1}{2}+\frac{4}{3}\ep}_{q(\ep), q(\ep),2} \cap \dot{B}^{-\frac{1}{2}}_{4,4,2}(I_T)} \leq \eta.
 \end{align}
 Increasing $T$ if necessary, Proposition~\ref{prop:FinalStateDeterministic} now provides a unique solution $(u,v) \in C(I_T, H^1(\R^3)) \times C(I_T, L^2(\R^3))$ of~\eqref{eq:ZakharovSystemFirstOrder} such that
 \begin{align*}
 	\| u - \eul^{\imu t \Delta} u_+^\omega\|_{X_T^\sigma}  + \| v - \eul^{\imu \alpha t |\nabla|} v_+ \|_{Y_T^\sigma} \leq r.
 \end{align*}
 Now assume that there are two solutions $(u_1, v_1), (u_2, v_2) \in C(I_T, H^1(\R^3)) \times C(I_T, L^2(\R^3))$ of~\eqref{eq:ZakharovSystemFirstOrder} satisfying
	\begin{align*}
		\| u_i - \eul^{\imu t \Delta} u_+^\omega\|_{X_T^\sigma} < \infty \qquad \text{and} \qquad \|v_i - \eul^{\imu \alpha t |\nabla|} v_+^\omega\|_{Y_T^\sigma} < \infty
	\end{align*}	 
   for $i = 1,2$. Setting $\nlschr^i = u_i - \eul^{\imu t \Delta} u_+^\omega$ and $\nlwave^i = v_i - \eul^{\imu \alpha t |\nabla|} v_+^\omega$, we thus find $r' \geq r$ such that $\|\nlschr^i\|_{X_T^\sigma} + \|\nlwave^i\|_{Y_T^\sigma} \leq r'$. Take $\eta' = \eta(r')$ from Proposition~\ref{prop:FinalStateDeterministic}. Employing dominated convergence again, we get a time $T_1 \geq T$ such that~\eqref{eq:CheckingAssumptionForDeterministicFinalState} holds with $T$ and $\eta$ replaced by $T_1$ and $\eta'$. Proposition~\ref{prop:FinalStateDeterministic} thus provides a time $T_2 \geq T_1$ such that $(u_1, v_1) = (u_2, v_2)$ on $I_{T_2}$. Using the estimates from Section~\ref{sec:MultilinearEstimates} and dominated convergence, it is now easy to see that this equality extends to $I_T$. \hfill $\qed$
 
 \vspace{1em}
 
Corollary~\ref{cor:GlobalSolution} follows from Theorem~\ref{thm:MainResult}, conservation of energy and Schr{\"o}dinger mass, and variational estimates from~\cite{GNW2013}.

\vspace{1em}

\emph{Proof of Corollary~\ref{cor:GlobalSolution}:} Let $\tilde{\Omega}$ be the set with measure $1$ provided by Theorem~\ref{thm:MainResult} such that for every $\omega \in \tilde{\Omega}$ there is $T \geq 1$ such that the final-state problem has a unique solution $(u,v)$ on $[T,\infty)$ and~\eqref{eq:CheckingAssumptionForDeterministicFinalState} is satisfied. By the local wellposedness theory from Proposition~3.1 in~\cite{GTV1997} we can extend $(u,v)$ in $X^{1,\frac{1}{2}+\delta} \times X^{0,\frac{1}{2} + \delta}$ for sufficiently small $\delta > 0$ in negative time direction to a maximal solution. We denote the maximal interval of existence of the extended solution by $I$. 

Using~\eqref{eq:CheckingAssumptionForDeterministicFinalState}, Lemma~\ref{lem:SobolevSchroedingerWave}, and interpolation, we infer that $\lim_{t \rightarrow \infty} \|e^{\imu t \Delta} u_+^\omega\|_{L^4} = 0$ for every $\omega \in \tilde{\Omega}$.
Since $(u,v)$ scatters in the energy space with final state $(u_+^\omega, v_+^\omega)$, we thus get
\begin{align*}
	E_Z(u,v) &= \lim_{t \rightarrow \infty} E_Z(e^{\imu t \Delta} u_+^\omega, e^{\imu \alpha t |\nabla|} v_+^\omega) = \frac{1}{2}\|\nabla u_+^\omega\|_{L^2}^2 + \frac{1}{4} \|v_+^\omega\|_{L^2}^2, \\
	M(u) &= M(u_+^\omega) = \frac{1}{2} \|u_+^\omega\|_{L^2}^2,
\end{align*} 
where we also exploited the Sobolev embedding $H^1(\R^3) \hookrightarrow L^4(\R^3)$.
Consequently, the assumption implies $E_Z(u,v) M(u) < E_S(Q) M(Q)$.
Moreover, we have $\lim_{t \rightarrow \infty} K(u(t)) = \|\nabla u_+^\omega\|_{L^2}^2$, where $K$ is the functional
\begin{align*}
 		K(\varphi) = \int_{\R^3} |\nabla \varphi|^2 - \frac{3}{4} |\varphi|^4 \dd x
\end{align*}
from~\cite{GNW2013}. Lemma~2.2 and Corollary~2.3 from~\cite{GNW2013} now imply $I = \R$.

Concerning uniqueness, we first note that the argument from the proof of Theorem~\ref{thm:MainResult} extends to any interval $[T',\infty) \subseteq [1,\infty)$. This means that if there is a solution $(u,v)$ which exists and satisfies~\eqref{eq:UniquenessCondition} on $[T',\infty)$, any other solution $(\tilde{u}, \tilde{v})$ on $[T',\infty)$ satisfying~\eqref{eq:UniquenessCondition} has to coincide with $(u,v)$ on $[T',\infty)$. The unconditional uniqueness result from~\cite{MN2009} for the initial value problem yields the uniqueness of the extension from $[T',\infty)$, so that the uniqueness assertion follows. \hfill $\qed$

\begin{rem}
	\label{rem:MeasureGlobal}
	We observe that in the case of small Schr{\"o}dinger data, we can give a lower bound on the measure of the set of those $\omega$ for which the nonlinear solution scattering with final state $(u_+^\omega, v_+^\omega)$ is global. Let $\tilde{\Omega}$ be as in the proof of Corollary~\ref{cor:GlobalSolution}. Using the argument from~\cite{BC1996}, there is a number $\eta > 0$ such that for all $\omega \in \tilde{\Omega}$ with $\|u_+^\omega\|_{H^1} \leq \eta$ the solution $(u,v)$ provided by Theorem~\ref{thm:MainResult} is global.
	
	Lemma~\ref{lem:LargeDeviationEstimate} and Corollary~\ref{cor:EstimatePkpsill2kl2lLp} show
	\begin{align*}
		\|u_+^\omega\|_{L^\beta_\omega H^1_x} \lesssim \sqrt{\beta} \|u_+\|_{H^1}
	\end{align*}
	for all $\beta \geq 2$. Lemma~\ref{lem:MeasureFromLargeDeviation} thus shows that there are constants $C,c > 0$ such that
	\begin{align*}
		\PP(\omega \in \tilde{\Omega}\colon \|u_+^\omega\|_{H^1} \leq \eta ) \geq 1 - C \eul^{-c \eta^2 \|u_+\|_{H^1}^{-2}}.
	\end{align*}
\end{rem}

\appendix
\section{Proof of Lemma~\ref{lem:MismatchEstimates}}
\label{app:ProofOfMismatch}
 
For the sake of completeness, we provide the details of the proof of Lemma~\ref{lem:MismatchEstimates}. We use the same arguments as the authors in~\cite{B2019, DLM2019}, where similar estimates have appeared.

\emph{Proof of Lemma~\ref{lem:MismatchEstimates}.}	We start with the proof of~\eqref{eq:MismatchSpaceFrequencySpace}. The estimate is trivial if $|l - l'| \leq 7$. So take $l, l' \in \Z^3$ with $|l - l'| \geq 8$ and $f \in L^p(\R^3)$. For $x \in \supp \psi_l, y \in \supp \psi_{l'}$, we then have $|l - l'| \leq 2 |x - y|$ so that
 	\begin{align*}
  	&|\psi_l P_k (\psi_{l'} f)(x)| = \Big| \psi_l(x) \int_{\R^3} \check{\rho}_k(x-y) \psi_{l'}(y) f(y) \dd y \Big| \\
  	 &\lesssim_D  \int_{\R^3} |\psi_l(x)| 2^{3k} \langle l - l' \rangle^{-D} \langle 2^k(x-y)\rangle^D |\check{\rho}_0(2^k(x-y))| |\psi_{l'}(y) f(y)| \dd y,
  \end{align*}
  where the implicit constant is independent of $k$ as $k \geq 1$. Since $\check{\rho}_0$ is a Schwartz function, Young's inequality further yields
  \begin{align*}
  	\| \psi_l P_k (\psi_{l'} f)\|_{L^p} \lesssim_D  \langle l - l' \rangle^{-D} \| \langle \cdot \rangle^D \check{\rho}_0\|_{L^1} \|f\|_{L^p} \lesssim_D \langle l - l' \rangle^{-D} \|f\|_{L^p}.
  \end{align*}
  Replacing $\rho_k$ by $\varphi$ in the above reasoning, where $\varphi(\xi) = \eta_0(|\xi|)$ for all $\xi \in \R^3$, we also obtain the second estimate in~\eqref{eq:MismatchSpaceFrequencySpace}.
  
  Next fix $l \in \Z^3$ and take $j,k \in \N$ with $|j - k| > 4$. Without loss of generality we can assume that $k > j$. Using the support restrictions of $\rho_k$ and $\rho_j$, we see that
 	\begin{align*}
 		\cF  P_k (\psi_l P_j g)(\xi) 
 		&= \int_{\R^3} \rho_k(\xi) \eul^{-\imu l (\xi - \eta)} \hat{\psi}(\xi - \eta) \rho_j(\eta) \hat{g}(\eta) \dd \eta \\
 		&= \int_{\R^3} \rho_k(\xi) \eul^{-\imu l (\xi - \eta)} \rho_{\geq k - 3}(\xi - \eta) \hat{\psi}(\xi - \eta)\rho_j(\eta) \hat{g}(\eta) \dd \eta
 	\end{align*}
for all $\xi \in \R^3$ and any Schwartz function $g$, where we used hat $j \leq k - 5$. Hence,
\begin{align*}
	P_k (\psi_l P_j f) = P_k ((P_{\geq k-3} \psi_l ) P_j f).
\end{align*}
Exploiting that $\hat{\psi}$ is a Schwartz function, we thus get
 \begin{align*}
 	\|P_k (\psi_l P_j f)\|_{L^p} &\lesssim \|P_{\geq k-3} \psi_l \|_{L^\infty} \|P_j f\|_{L^p} \\
 	&\lesssim_D 2^{-2D k} \| \rho_{\geq k-3}(\xi) \eul^{- \imu l \xi} |\xi|^{2D} \hat{\psi}(\xi)\|_{L^1_\xi} \|f\|_{L^p} \lesssim_D 2^{-2D k} \|f\|_{L^p}, 
 \end{align*}
 where the implicit constant is independent of $l$. Replacing $P_j$ by $P_{\leq k - 5}$, the same arguments also yield~\eqref{eq:MismatchFrequencySpaceFrequency2}. \hfill $\qed$

\subsection*{Acknowledgment}
 
 I want to thank Sebastian Herr for invaluable advice on this project, instructive discussions, and helpful feedback.
 
 Moreover, financial support by the German Research Foundation (DFG) through the CRC
1283 “Taming uncertainty and profiting from randomness and low regularity in analysis, stochastics and
their applications” is acknowledged.

\end{document}